\documentclass[12pt]{iopart}
\usepackage{}
\usepackage{graphicx}
\usepackage{algorithm}
\usepackage{caption}
\usepackage{subcaption}
\usepackage{epstopdf}
\usepackage{iopams}
\usepackage{amsthm, amsfonts,amssymb}
\usepackage{placeins}
\usepackage{enumitem}
\usepackage{cleveref}
\usepackage{textcomp}
\graphicspath{ {./Figures/} }

\newtheorem{thm}{Theorem}[section]
\newtheorem{Lem}[thm]{Lemma}

\newtheorem{prop}[thm]{Proposition}

\newtheorem{rem}{Remark}

\theoremstyle{definition}
\newtheorem{defn}{Definition}[section]
\theoremstyle{remark}

\begin{document}

\title[ Period tripling and quintupling renormalizations below $C^2$ space]{Period tripling and quintupling renormalizations below $C^2$ space }

\author{Rohit Kumar$^1$, V.V.M.S. Chandramouli$^2$ }

\address{$^{1,2}$ Department of Mathematics, Indian Institute of Technology Jodhpur, \;\;\; Rajasthan, India-342037. }
\eads{ $^{1}$  \mailto kumar.30@iitj.ac.in, $^{2}$  \mailto chsarma@iitj.ac.in.}
\vspace{10pt}
\begin{indented}
\item[]October 2020
\end{indented}

\begin{abstract}
In this paper, we explore the period tripling and period quintupling renormalizations below $C^2$ class of unimodal maps.
We show that for a given proper scaling data there exists a renormalization fixed point on the space of piece-wise affine maps which are infinitely renormalizable. Furthermore, we show that this renormalization fixed point is extended to a $C^{1+Lip}$ unimodal map, considering the period tripling and period quintupling combinatorics. Moreover, we show that there exists a continuum of fixed points of renormalizations by considering a small variation on the scaling data. Finally, this leads to the fact that the tripling and quintupling renormalizations acting on the space of $C^{1+Lip}$ unimodal maps have unbounded topological entropy.
\end{abstract}

%
\vspace{2pc}
\eqnobysec

\section{Introduction}

The concept of renormalization arises in many forms though Mathematics and Physics. Renormalization is a technique to describe the dynamics of a given system at a small spatial scale by an induced dynamical system in the same class. Period doubling renormalization operator was introduced by M. Feigenbaum \cite{Fe}, \cite{Fe2} and by P. Coullet and C. Tresser \cite{CT}, to study asymptotic small scale geometry of the attractor of one dimensional systems which are at the transition from simple to chaotic dynamics.

 The hyperbolicity of unique renormalization fixed point has been shown by O. Lanford \cite{LanF} for period doubling operator and later M. Lyubich \cite{LyuM} gave the generalization to the other sort of renormalizations, in the holomorphic context. Then it was shown by A. Davie \cite{Davie}, the renormalization fixed point is also hyperbolic in the space of $C^{2+\alpha}$ unimodal maps with $\alpha > 0.$ These results further extended by E. de Faria, W. de Melo and A. Pinto \cite{FMP} to a more general type of renormalization, using the results of M. Lyubich \cite{LyuM}. 
 Later, it was extended by Chandramouli, Martens, de Melo, Tresser \cite{CMMT}, to a new class denoted by $C^{2+|\cdot|}$ which is bigger than $C^{2+\alpha}$ (for any positive $ \alpha \leq 1$), in which the period doubling renormalization converges to the analytic  generic fixed point proving it to be globally unique. Furthermore it was shown that below $C^2,$ uniqueness is lost and period doubling renormalization operator has infinite topological entropy. 

One dimensional systems which arise from the real time applications are usually not smooth. In dissipative systems the states are groups in so-called stable manifolds,  different states in such a stable manifold have the same future. The packing of the stable manifold usually does not occur in a smooth way. For example, the Lorenz flow is a flow on three dimensional space approximating a convection problem in fluid dynamics. The stable manifolds are two dimensional surfaces packed in a non smooth foliation. This flow can be understood by a map on the interval whose smoothness is usually below $C^2$.

 In this work, we describe the construction of renormalization fixed point below $C^2$ space by considering period tripling and quintupling combinatorics. In the case of period tripling renormalization, for a given proper scaling data, we construct a nested sequence of affine pieces whose end-points lie on the unimodal map  and shrinking down to the critical point. Also, we prove that the period tripling renormalization operator defined on the space of piece-wise affine infinitely renormalizable maps, has a fixed point, denoted by $f_{s^*},$ corresponding to a given proper scaling data. In the next subsection~\ref{extsn}, we describe the extension of this renormalization fixed point $f_{s^*}$ to a $C^{1+Lip}$ unimodal map. Furthermore, in subsection~\ref{entropy}, we show that the topological entropy of renormalization defined on the space of $C^{1+Lip}$ unimodal maps has unbounded entropy. In subsection~\ref{evar}, we consider an $\epsilon-$variation on the scaling data and describe the existence of continuum of renormalization fixed points. In section ~\ref{sec3}, we describe the construction of period quintupling renormalization fixed point $g_{s^*}$ and its extension to $C^{1+Lip}$ space. Finally, we show that the geometry of the invariant Cantor set of the map $g_{s^*}$ is more complex than the geometry of the invariant Cantor set of $f_{s^*}$. \\
\noindent We recall some basic definitions.  Let $I = [a,b]$  be a closed interval. \\

\ A map $f : I  \rightarrow I$ , is a $C ^1$ map  with unique critical point $c \in I$, is called \textit{unimodal} map. \\

\ Let $c$ be a critical point of a unimodal map $f$ has a \textit{quadratic tip} if there exists a sequence $\{y_n\}$ approaches to $c$ and a constant $l > 0$ such that  
 $$\lim_{n\to\infty} \frac{f(y_n)-f(c)}{(y_n-c)^2}  = -l.$$      
   
 A unimodal map $f $ is called \textit{renormalizable} if there exists a proper subinterval $J$ of $I$ and a positive integer $n$ such that \\
(1)\;  $f^i(J), \; i = 0,1,....., n-1$ have no pairwise interior intersection, \\
(2)\;  $f^n(J) \subset J.$ \\ 
\noindent Then $f^n: J \rightarrow J$ is called a \textit{renormalization} of $f.$ \\ 

  A map $f: I \rightarrow I$ is \textit{infinitely renormalizable} map if there exists an  infinite sequence $\{I_n\}_{n = 0}^\infty$ of nested intervals and an infinite sequence $\{k(n)\}_{n = 0}^\infty$ of positive integers such that ${f^{k(n)}|_{I_n}} : I_n \rightarrow I_n$ are renormalizations of $f$ and the length of $I_n$ tends to zero as $n \rightarrow \infty.$ \\

Let $U$ be the set of unimodal maps and $U_0 \subset U $ contains the set of period doubling renormalizable unimodal maps. Let $f  \in U_0.$ Then, the period doubling renormalization operator $$R : U_0 \rightarrow U$$ is defined by 
$$Rf(x) = h^{-1} \circ f^2 \circ h(x),$$ where $h:[0,1] \rightarrow J$ is the orientation reversing affine homeomorphism.   
The map $Rf$ is again a unimodal map. The set of \textit{period doubling infinitely renormalizable} maps is denoted by 
$$W = \bigcap\limits_{n \geq 1} R^{-n}(U_0).  $$

\noindent In the next section, we construct a fixed point of period tripling renormalization.

\section{Period tripling renormalization} 

\subsection{Piece-wise affine renormalizable maps}\label{p2}


 Let us consider a family of maps $\mathcal{U}_c : [0,1] \rightarrow [0,1] $ defined by $$ \mathcal{U}_c(x) = 1-\Big|\frac{x-c}{1-c}\Big|^{\alpha} .$$ Where $\alpha > 1$ is the critical exponent and $c \in [0,\frac{1}{2}]$ is the critical point.
\\ 
\noindent In particular, we consider $\alpha =2$ and the subclass of unimodal maps $ \mathcal{U}_c(x)$ is denoted by $u_c(x).$\\ 

\noindent  Define an open set $$ T_k  = \left\{ (s_1, s_2, s_3, ....., s_k) \in \mathbb{R}^k : s_1,s_2,s_3, .....,s_k >0, \; \sum\limits_{i = 1}^k s_i < 1 \right\}, \; k \in \mathbb{N}.$$ For $k = 3,$ each element $ (s_1, s_2, s_3) $ of $ T_3 $ is called a scaling tri-factor. \\
 Define affine maps $\tilde{s}_1,$ $\tilde{s}_2$ and $\tilde{s}_3$ induced by a scaling tri-factor,  
\begin{eqnarray*}
\tilde{s}_1  : [0,1]\rightarrow [0,1]   \\
\tilde{s}_2  : [0,1]\rightarrow [0,1]   \\
\tilde{s}_3  : [0,1]\rightarrow [0,1]   
\end{eqnarray*}
 defined by 
\begin{eqnarray*}
 \tilde{s}_1(t) &=  s_1(1-t) \\ 
 \tilde{s}_2(t) &=  u_c(0)+ s_2(1-t) \\
 \tilde{s}_3(t) &= 1-s_3(1-t).
\end{eqnarray*}

\ A function $ s : \mathbb{N} \rightarrow T_k $ is called a   scaling data. For $k = 3,$ for all $ n \in \mathbb{N}$, the scaling tri-factor $ s(n) = (s_1(n),s_2(n), s_3(n)) \in T_3 $ induces a triplet of affine maps $ (\tilde{s}_1(n) , \tilde{s}_2(n) , \tilde{s}_3(n))$. Now, for each $n \in \mathbb{N},$ we define the following intervals:
\begin{eqnarray*}
I_1^n = \tilde{s}_2(1)\circ \tilde{s}_2(2)\circ \tilde{s}_2(3) \circ.....\circ \tilde{s}_2(n-1)\circ \tilde{s}_1(n)([0,1]), \\
I_2^n = \tilde{s}_2(1)\circ \tilde{s}_2(2)\circ \tilde{s}_2(3) \circ.....\circ \tilde{s}_2(n-1)\circ \tilde{s}_2(n)([0,1]), \\ 
I_3^n = \tilde{s}_2(1)\circ \tilde{s}_2(2)\circ \tilde{s}_2(3) \circ.....\circ \tilde{s}_2(n-1)\circ \tilde{s}_3(n)([0,1]).
\end{eqnarray*}

\begin{defn}
\ A scaling data is said to be proper if $$ d(s(n),\partial T_k) \geq \epsilon, \;\;\;\;\;   \textrm{for some} \;\;\; \epsilon > 0.$$ Also, we have $$ |I_j^n| \leq (1-\epsilon)^n,  \;\;\; \textrm{for} \;\;\;  n \geq 1 \;\;\; \textrm{and} \;\;\; j = 1,2,3. $$
\end{defn}

\noindent Given a proper scaling data, define $$\{c\} = \bigcap\limits_{n \geq 1} I_2^n. $$

\noindent  A proper scaling data $ s : \mathbb{N} \rightarrow T_3 $ induces the set $ D_s  = \bigcup\limits_{n \geq 1} (I_1^n \cup I_3^n) .$   Consider a map $$ f_s : D_s \rightarrow [0,1] $$ such that $f_s|_{I_1^n}$ and $f_s|_{I_3^n}$ are the affine extensions of  $u_c|_{\partial I_1^n}$ and $u_c|_{\partial I_3^n}$ respectively. 

\FloatBarrier
\begin{figure}[h]
\centering
{\includegraphics [width=86mm]{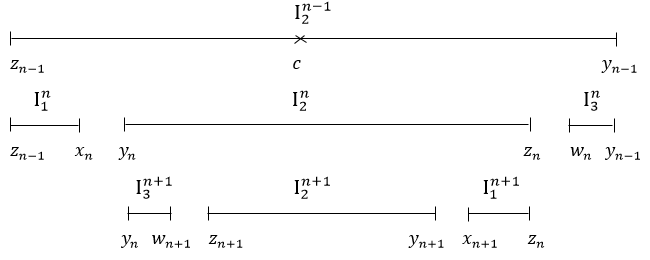}}
\caption{Intervals of next generations}
\label{fig:nexgen}
\end{figure} 
\FloatBarrier

\noindent From Figure~\ref{fig:nexgen}, the end points of the intervals at each level are labeled by
 
\noindent  $$z_0 = 0,\; y_{0} = 1,\; I_2^0 = [0,1] $$ and for $ n \geq 1$ 
\begin{center}
$ x_n =  \partial I_1^n \backslash \partial I_2^{n-1} $\\
$ y_{2n-1} = min\{ \partial I_2^{2n-1} \} $\\
$ y_{2n} = max\{ \partial I_2^{2n} \} $\\
$ z_{2n-1} = max\{ \partial I_2^{2n-1} \} $\\
$ z_{2n} = min\{ \partial I_2^{2n} \} $\\
$ w_n =  \partial I_3^n \backslash \partial I_2^{n-1}. $
\end{center}

\begin{defn}
\ For a given proper scaling data $s : \mathbb{N} \rightarrow T_3, $ a map $f_s$ is said to be \textit{period tripling infinitely renormalizable} if for $n \geq 1, $ 
\begin{itemize}
\item[(i)]  $[f_s(y_{n}),1]$ is the maximal domain containing $1$ on which $f_s^{3^n-1}$ is defined affinely and $[0, f_s^2(y_{n})]$ is the maximal domain containing $0$ on which $f_s^{3^n-2}$ is defined affinely, 
\item[(ii)] $f_s^{3^n-1}([f_s(y_{n}),1]) = I_2^n,$ \\ $f_s^{3^n-2}([0, f_s^2(y_{n})]) = I_2^n.$
\end{itemize}
\end{defn}

\noindent The combinatorics for period  tripling renormalization is shown in Figure \ref{fig:ptr}.

\begin{figure}[!htb]
\centering
{\includegraphics [width=95mm]{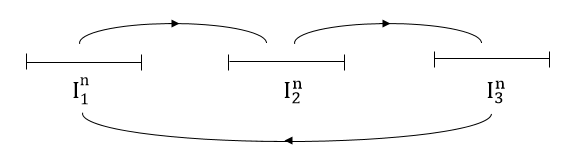}}
\caption{Period triple interval combinatorics ($ I_2^n \rightarrow I_3^n \rightarrow I_1^n \rightarrow I_2^n).$ }
 \label{fig:ptr}
\end{figure} 

\noindent Let $U_\infty$ be the collection of period tripling infinitely renormalizable maps. \\
 \noindent Let $f_s \in U_\infty$ be given by the proper scaling data $s : \mathbb{N} \rightarrow T_3 $ and
  define $$ {\tilde{I}}_2^n  = [0, min\{u_c^{-1}( z_{n})\}] = [0, min\{f_s^{-1}( z_{n})\}] ,$$  
  where $ u_c^{-1}(x)$ denotes the preimage(s) of $x$ under $u_c$ and $$ {\hat{I}}_2^n  = [u_c(y_{n}),1] = [f_s(y_{n}),1].$$ 
   Let $$ h_{s,n} : [0,1] \rightarrow [0,1] $$ be defined by $$h_{s,n} = \tilde{s}_2(1)\circ \tilde{s}_2(2) \circ \tilde{s}_2(3)\circ..... \circ \tilde{s}_2(n). $$
Furthermore, let $${\tilde{h}}_{s,n} : [0,1] \rightarrow {\tilde{I}}_2^n \;\; \textrm{and} \;\; {\hat{h}}_{s,n} : [0,1] \rightarrow {\hat{I}}_2^n $$ be the affine orientation preserving homeomorphisms. 

\noindent Then define $$R_nf_s: h_{s,n}^{-1}(D_s) \rightarrow [0,1]$$  by $$R_nf_s(x) =  \left\{\begin{array}{ll}
R_n^{-}f_s(x), \;\;\; \;  \; \textrm{if} \; x \in h_{s,n}^{-1}(I_1^n) \\ R_n^{+}f_s(x), \;\;\;  \;  \; \textrm{if} \; x \in h_{s,n}^{-1}(I_3^n)
\end{array}
\right.$$ \\
where, $$R_n^{-}f_s: h_{s,n}^{-1}(\mathop{\cup}\limits_{n \geq 1}I_1^n) \rightarrow [0,1] $$\;\; and $$ R_n^{+}f_s: h_{s,n}^{-1}(\mathop{\cup}\limits_{n \geq 1}I_3^n) \rightarrow [0,1]$$ are defined by $$R_n^{-}f_s(x) = {\tilde{h}}_{s,n}^{-1} \circ f_s^{-1} \circ h_{s,n}(x)$$ $$R_n^{+}f_s(x) = {\hat{h}}_{s,n}^{-1} \circ f_s \circ h_{s,n}(x), $$ which are illustrated in Figure~\ref{fig:renomop}.

\begin{figure}[!htb]
\centering
{\includegraphics [width=120mm]{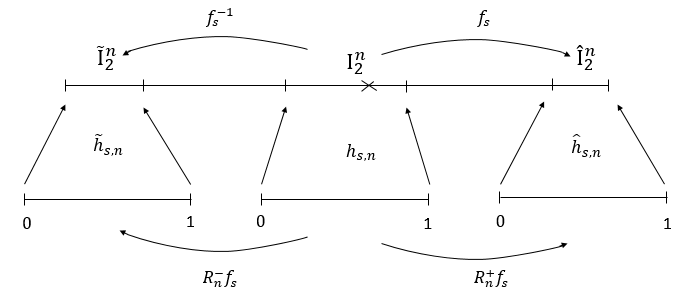}}
\caption{ }
 \label{fig:renomop}
\end{figure}

\noindent Let $\sigma: T_k^{\mathbb{N}} \rightarrow T_k^{\mathbb{N}}$ be the shift map $ \sigma(s)(n) = s(n+1). $  The above construction gives us the following result,

\begin{Lem}\label{lem1}
\ Let $s : \mathbb{N} \rightarrow T_3$ be proper scaling data such that $f_s$ is infinitely renormalizable. Then $$R_n{f_s} = f_{\sigma^{n}(s)}. $$
\end{Lem}

\noindent Let $f_s$ be infinitely renormalization, then for $n \geq 0,$ we have $$ f_s^{3^n}: D_s \cap I_2^n \rightarrow I_2^n $$ is well defined. \\
\ Now we define the renormalization $R : U_\infty \rightarrow U_\infty$ as
$$ Rf_s = h_{s,1}^{-1} \circ f_s^3 \circ h_{s,1} .$$

\noindent The maps $f_s^{3^n-2} : {\tilde{I}}_2^n \rightarrow I_2^n $ and $f_s^{3^n-1} : {\hat{I}}_2^n \rightarrow I_2^n $ are the affine homeomorphisms whenever $f_s \in U_\infty$. Then we have

\begin{Lem}\label{lem2}
\ $R^n{f_s} : D_{\sigma^{n}(s)} \rightarrow [0,1]$ and $R^n{f_s} = R_n{f_s}.$
\end{Lem}

\noindent The~\cref{lem1} and~\cref{lem2}  give the following result.

\begin{prop} \label{thm1}
\ There exists a map $f_{s^*} \in U_\infty,$ where $s^*$ is characterized by $$Rf_{s^*} = f_{s^*}. $$ In particular, $U_\infty  = \{f_{s^*}\}.$
\end{prop}
\begin{proof}
\ Consider $s : \mathbb{N} \rightarrow T_3$ be proper scaling data such that $f_s  $ is an infinitely renormalizable. Let $c_n$ be the critical point of $f_{\sigma^n(s)}.$ Then

\FloatBarrier
\begin{figure}[h]
\centering
{\includegraphics [width=90mm]{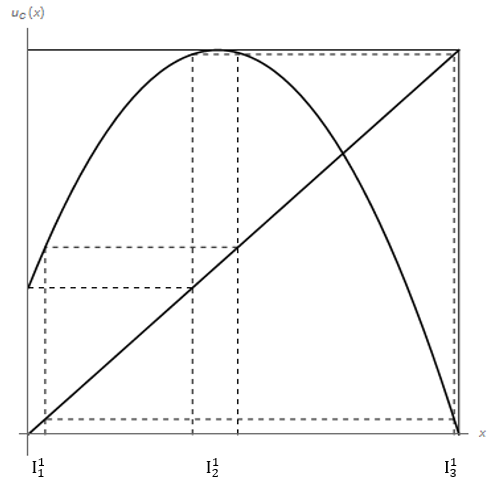}}
\caption{Period three cobweb diagram}
\label{fig:length}
\end{figure} 
\FloatBarrier

\FloatBarrier
\begin{figure}[h]
\centering
{\includegraphics [width=90mm]{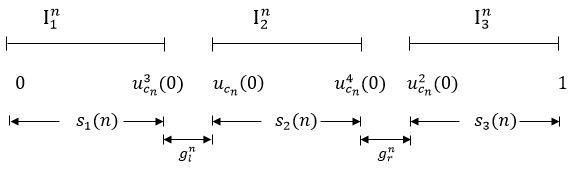}}
\caption{Length of the intervals gaps}
\label{fig:ilength1}
\end{figure} 
\FloatBarrier

\noindent from Figures \ref{fig:length} and~\ref{fig:ilength1}, we have the following 
\begin{eqnarray}
 u_{c_n}^2(0) &= 1- s_3(n) \label{eq01} \\ 
  u_{c_n}^3(0) &= u_{c_n}(1-s_3(n)) \nonumber  \\ &= s_1(n) \label{eq02} \\
 u_{c_n}^4(0) &= s_2(n) + u_{c_n}(0)  \label{eq03} \\
 c_{n+1} &= \frac{u_{c_n}^4(0)-c_n}{s_2(n)}.  \label{eq04}
\end{eqnarray}
 
\noindent Since $(s_1(n), s_2(n), s_3(n)) \in T_3,$ we have the following conditions 
\begin{eqnarray}
 s_i(n) >0,\; \textrm{for} \; i = 1,2,3  \label{eq5} \\
 s_1(n) +s_2(n) +s_3(n) & <1 \label{eq6} 
 \end{eqnarray}

\ As the intervals $I_i^n,$ for $i = 1,2,3,$ are pairwise disjoint, let $g_l^n$ and $g_r^n$ are the gaps  between $I_1^n \;\&\; I_2^n$ and $I_2^n \; \& \; I_3^n$ respectively, which are shown in Figure  \ref{fig:ilength1}. Then we have,
for $n \in \mathbb{N},$
\begin{eqnarray}
g_l^n = u_{c_n}(0)-u_{c_n}^3(0) \equiv G_l(c_n) >0  \label{g1} \\
g_r^n = u_{c_n}^2(0)-u_{c_n}^4(0) \equiv G_r(c_n) > 0 \label{g2} \\
 0 < c_n < \frac{1}{2} \label{eq6a}
\end{eqnarray}

\noindent  Solving Eqns.~(\ref{eq01}),~(\ref{eq02}),~(\ref{eq03}) and~(\ref{eq04}) by using Mathematica, we get

\begin{eqnarray}
 \fl s_1(n)  = 1- \frac{ (c_n - 8 c_n^2 + 21 c_n^3 - 25 c_n^4 + 17 c_n^5 - 6 c_n^6 + c_n^7)^2}{(1 - c_n)^{14}} \equiv S_1(c_n) \label{eq10} \\  
\fl s_2(n) =  \frac{c_n^2 (1-c_n)^{28} -\Big((-1+c_n)^{15}+(c_n - 8 c_n^2 + 21 c_n^3 - 25 c_n^4 + 17 c_n^5 - 6 c_n^6 + c_n^7)^2\Big)^2}{(1 - c_n)^{30}} \nonumber \\ \hspace{-1.5cm} \equiv S_2(c_n)   \label{eq11} \\
\fl s_3(n) = \frac{(-1 + 3 c_n - 2 c_n^2 + c_n^3)^2}{(1 - c_n)^6} \equiv S_3(c_n) \label{eq12} \\ 
\fl c_{n+1} = \frac{(1 - c_n)^{31} - ((1 - c_n)^{15} - (c_n - 8 c_n^2 + 21 c_n^3 - 25 c_n^4 + 
     17 c_n^5 - 6 c_n^6 + c_n^7)^2)^2}{ c_n^2 (1 - c_n)^{28} - ((1 - c_n)^{15} - (c_n - 8 c_n^2 + 21 c_n^3 - 
     25 c_n^4 + 17 c_n^5 - 6 c_n^6 + c_n^7)^2)^2}  \nonumber \\ \hspace{-1.5cm} \equiv \mathcal{R}(c_n)  \label{eq13}
\end{eqnarray}
 
\noindent The graphs of $S_1,\; S_2,\; S_3$ and $S_1+S_2+S_3$ are shown in Figures~\ref{fig:plt11} and \ref{fig:plt}.

\begin{figure}[ht]
  \centering
  \begin{subfigure}[b]{0.4\linewidth}
    \centering\includegraphics[width=140pt]{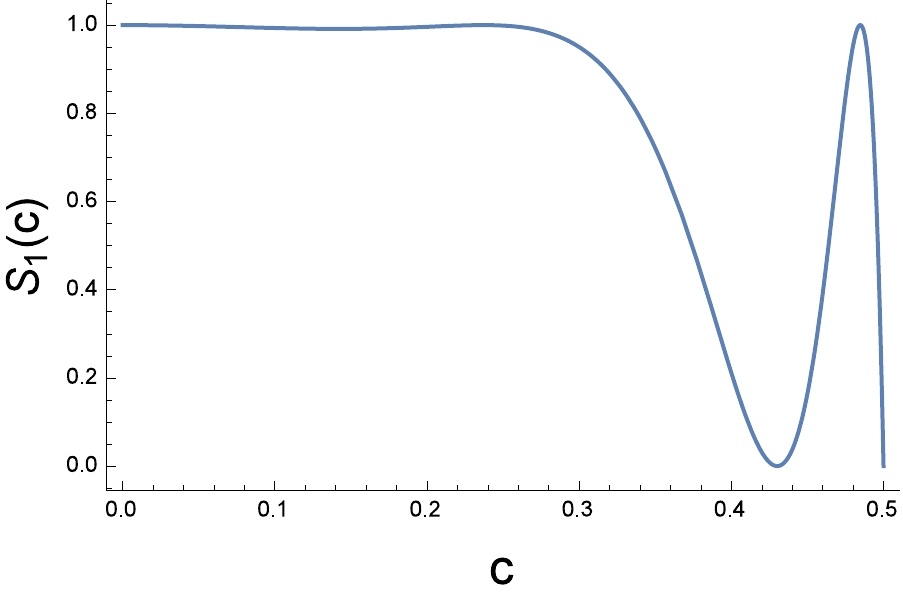}
    \caption{\label{fig:fig1}}
  \end{subfigure}%
  \begin{subfigure}[b]{0.4\linewidth}
    \centering\includegraphics[width=140pt]{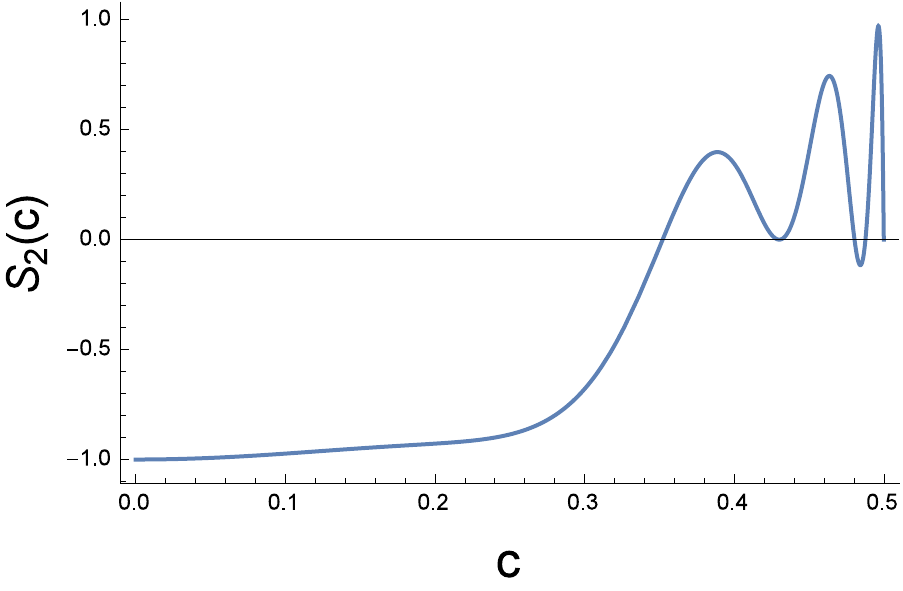}
    \caption{\label{fig:fig2}}
  \end{subfigure}
  \caption{(\subref{fig:fig1}) and (\subref{fig:fig2}) show the graphs of $S_1(c)$ and $S_2(c)$ respectively.}
  \label{fig:plt11}
  \end{figure}
  \begin{figure}
  \begin{subfigure}[b]{0.4\linewidth}
    \centering\includegraphics[width=140pt]{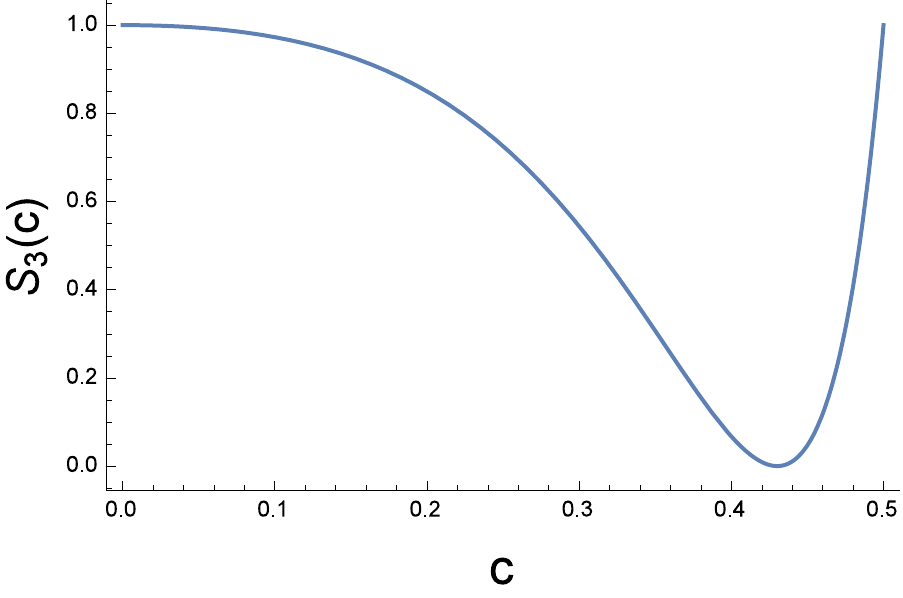}
    \caption{\label{fig:fig3}}
  \end{subfigure}%
  \begin{subfigure}[b]{0.4\linewidth}
    \centering\includegraphics[width=140pt]{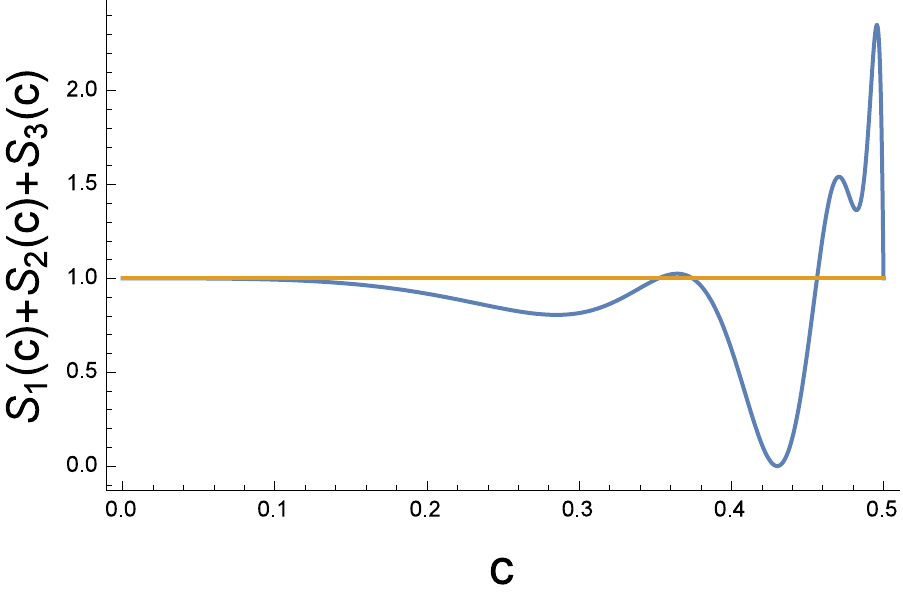}
    \caption{\label{fig:fig4}}
  \end{subfigure}
  \caption{ (\subref{fig:fig3}) and (\subref{fig:fig4}) show the graphs of $S_3(c)$ and $(S_1+S_2+S_3)(c)$ respectively.}
 
  \label{fig:plt}
\end{figure}               

\noindent Note that the conditions (\ref{eq5}). (\ref{g1}) and (\ref{g2}) gives the condition (\ref{eq6}) $$ 0<\sum_{i=1}^{3} s_i(n) < 1.$$
Therefore, the condition~(\ref{eq5}) together with (\ref{g1}), (\ref{g2}) and~(\ref{eq6a})  define the feasible domain $F_d$ to be: 
\begin{eqnarray}\label{eqfd}
\fl F_d = \Big\{ \; c \in \left(0,\; \frac{1}{2}\right) \; : \;  S_i(c) > 0 \; \textrm{for} \; i=1,2,3, \; G_l(c) >0 ,\; G_r(c)>0 \Big\} 
\end{eqnarray} 

To compute the feasible domain $F_d,$  we need to find subinterval(s) of $(0, 0.5)$ which satisfies the conditions of (\ref{eqfd}). By using Mathematica software, we employ the following command to obtain the feasible domain
 $$\textup{N[Reduce[$\{S_1(c) > 0,S_2(c) > 0,S_3(c) > 0,G_l(c) > 0,G_r(c) > 0,0<c<0.5\}$,c]]}. $$    This yields: 
\begin{eqnarray*}
\fl F_d = (0.398039... ,\;0.430159...) \cup (0.430159...,\;0.456310...) \equiv F_{d_1} \cup F_{d_2}.
\end{eqnarray*}
From the Eqn.(\ref{eq13}), the graphs of $\mathcal{R}(c)$ are plotted in the sub-domains $F_{d_1}$ and $F_{d_2}$ of $F_d$ which are shown in Figure \ref{fig:R}. 
 
\begin{figure}[ht]
  \centering
  \begin{subfigure}[b]{0.5\linewidth}
    \centering\includegraphics[width=130pt]{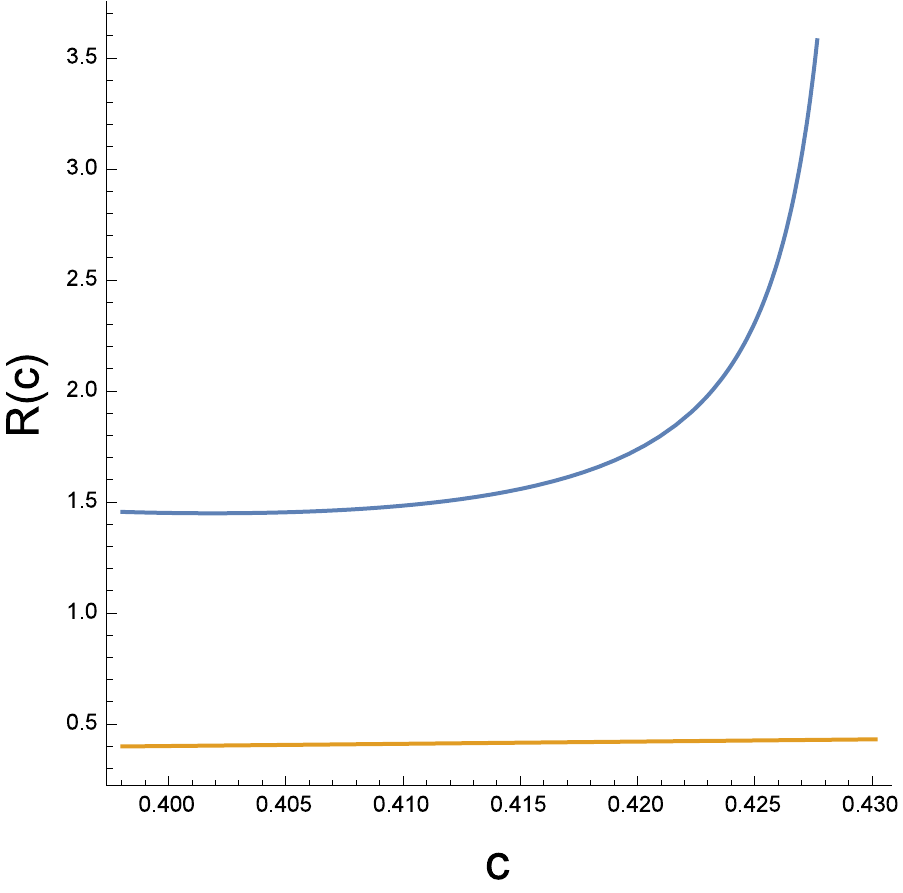}
   \subcaption{ $\mathcal{R}$ has no fixed point in $F_{d_1}.$ }
    \label{fig:picf1}
  \end{subfigure}%
  \begin{subfigure}[b]{0.5\linewidth}
    \centering\includegraphics[width=160pt]{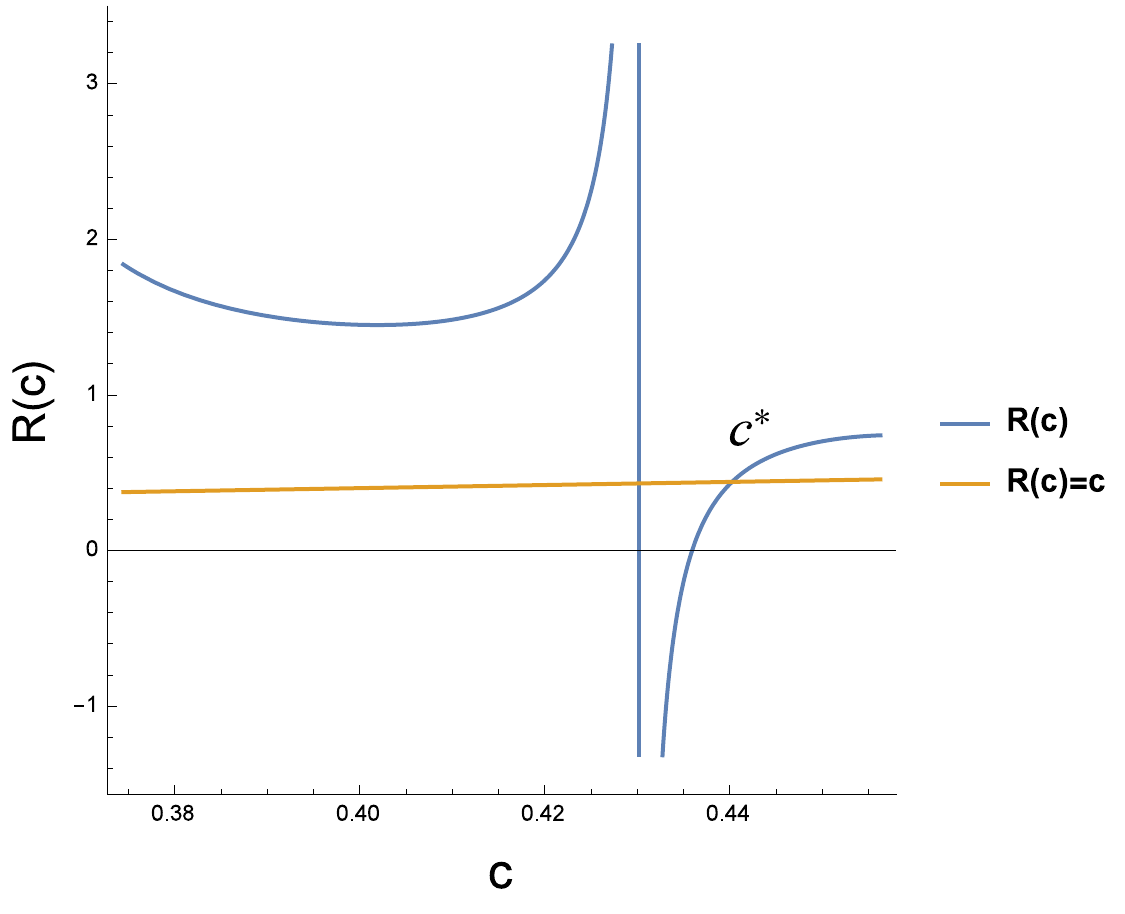}
    \subcaption{$\mathcal{R}$ has only one fixed point in $F_{d_2}.$ }
    \label{fig:picf2}
  \end{subfigure}
  \caption{The graph of $\mathcal{R} : F_d \rightarrow \mathbb{R}$ and the diagonal $\mathcal{R}(c) = c.$}
  \label{fig:R}
\end{figure} 
 
Clearly, the map $\mathcal{R}  : F_d \rightarrow \mathbb{R}$ is expanding in the neighborhood of the fixed point $c^*$ which is shown in Figure~\ref{fig:picf2}.  By using Mathematica, we compute the only fixed point $c^* = 0.440262... $ in $F_d$ such that $$\mathcal{R}(c^*) = c^*$$ corresponds to the infinitely renormalizable map $f_{s^*}.$ 

The graphs of piece-wise infinitely renormalizable map $f_{s^*}$ and zoomed part of $f_{s^*}$ are shown in the Figure~\ref{fig:replt1}.

\begin{figure}[ht]
  \centering
  \begin{subfigure}[b]{0.5\linewidth}
    \centering\includegraphics[width=160pt]{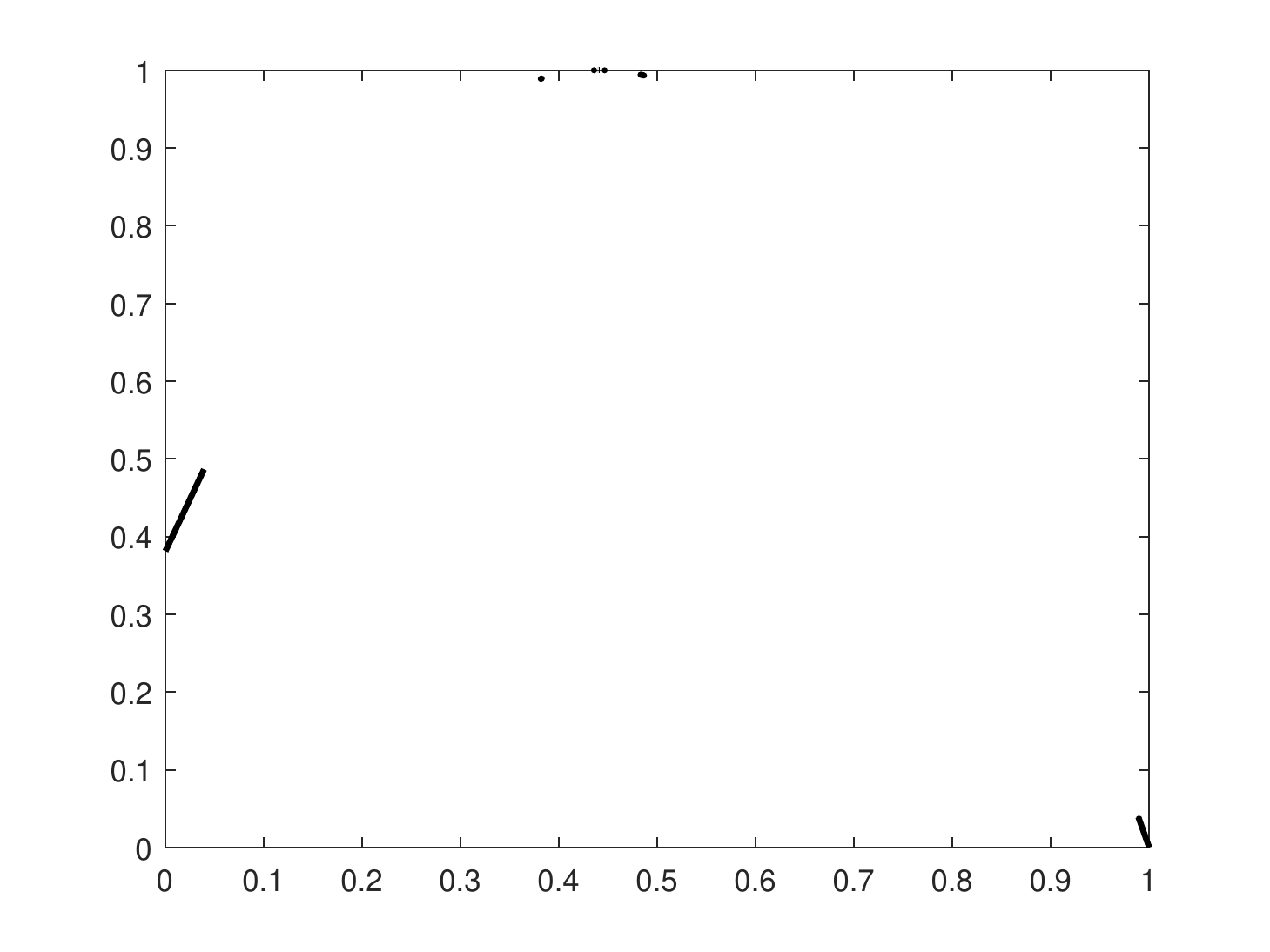}
   \subcaption{ $f_{s^*}$ in $[0,1].$ }
    \label{fig:pic01}
  \end{subfigure}%
  \begin{subfigure}[b]{0.5\linewidth}
    \centering\includegraphics[width=160pt]{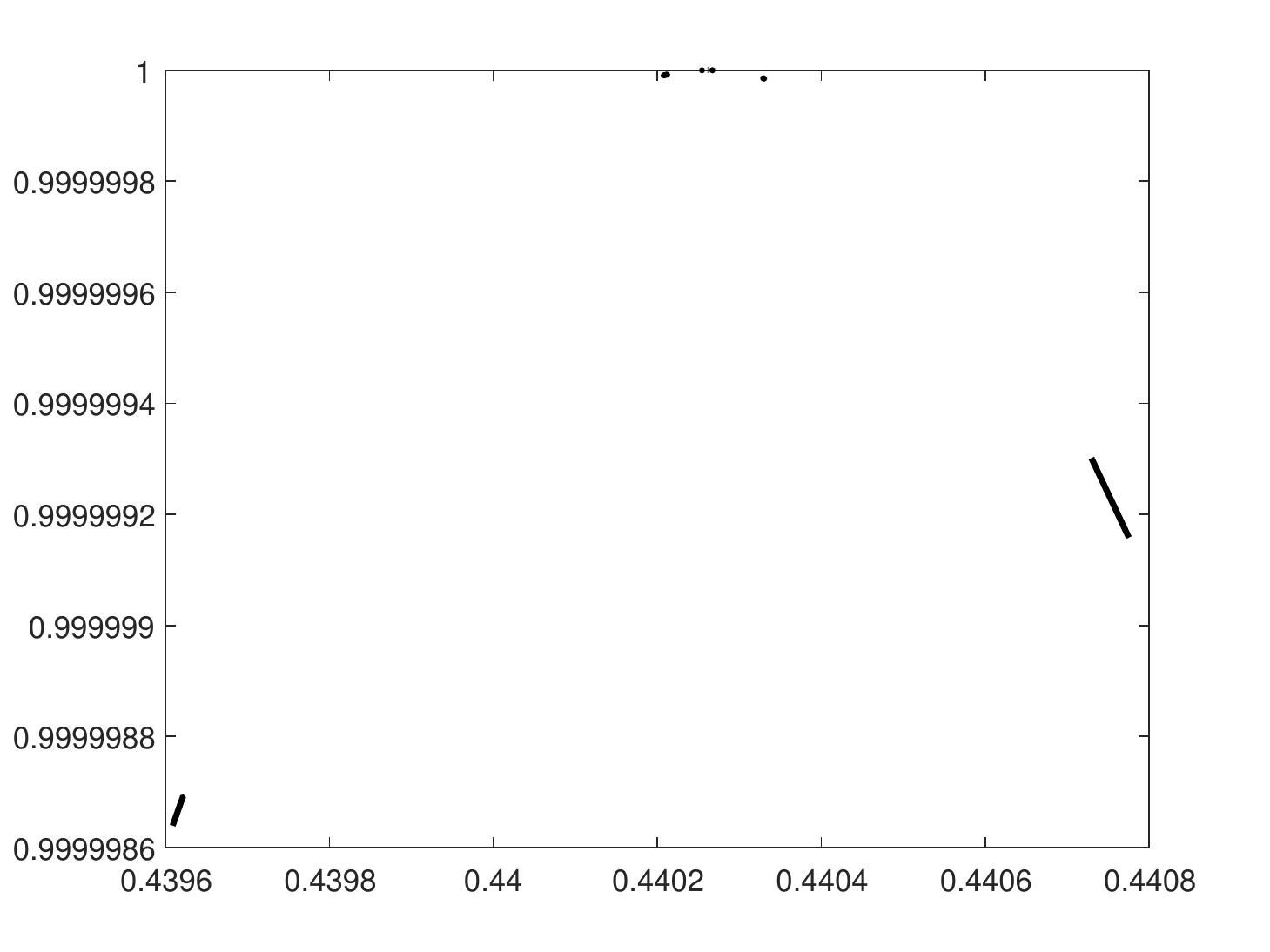}
    \subcaption{Zoomed part of  $f_{s^*} $ in $[0.4396,0.4408]$ }
    \label{fig:pic02}
  \end{subfigure}
  \caption{The graph of map $f_{s^*}.$}
  \label{fig:replt1}
\end{figure}
 
\noindent In other words, consider the scaling data $s^* : \mathbb{N} \rightarrow T_3,$ with 
 \begin{eqnarray*}
 s^*(n) &= (s_1^*(n),\; s_2^*(n),\; s_3^*(n) )  \\
  &= (u_{c^*}^3(0),\; u_{c^*}^4(0)-u_{c^*}(0),\; 1-u_{c^*}^2(0)).
 \end{eqnarray*}
Then $\sigma(s^* ) = s^* $ and using Lemma \ref{lem1} we have
   $$Rf_{s^*} = f_{s^*}. $$
\end{proof}


\begin{Lem}\label{lem24}
\ If $ f_{{s}^*}$ is the map with a proper scaling data $ {s}^* = (s_1^*, s_2^*, s_3^*),$ then we have
 $$(s_2^*)^2=s_3^*.$$
\end{Lem}
\begin{proof}
\ Let $\hat{I}_2^n = f_{{s}^*}(I_2^n)  = [f_{{s}^*}(y_{n}),1]$ and $\hat{I}_3^{n+1} = f_{{s}^*}(I_3^{n+1}).$ Then $f_{{s}^*}^{ 3^n-1} : \hat{I}_2^n \rightarrow I_2^n  $ is affine, monotone and onto. Further, by construction
$$ f_{{s}^*}^{ 3^n-1} : \hat{I}_2^{n+1} \rightarrow I_3^{n+1}. $$ Hence, $$\frac{|\hat{I}_2^{n+1}|}{|\hat{I}_2^n|} = s_3^*.$$ Therefore,
$|I_2^n| = (s_2^*)^n$ and $|\hat{I}_2^n| = (s_3^*)^n$. Since, $$f_{{s}^*}(y_n) = u_{{c}^*}(y_n).$$ This implies,
$$s_3^* = \frac{|\hat{I}_2^{n+1}|}{|\hat{I}_2^n|} = \left(\frac{{y_{n+1}}-c}{{y_{n}}-c}\right)^2  =\left( \frac{|{I}_2^{n+1}|}{|{I}_2^n|}\right)^2 = (s_2^*)^2.$$

\end{proof}

\begin{rem}\label{lem3}
\ A proper scaling data $(s_1^*, s_2^*, s_3^*)$ is satisfying the following properties, for all $n \in \mathbb{N},$
\begin{enumerate}[label=(\roman*)] \normalsize{ 
\item $ \frac{|I_2^{n+1}|}{|I_2^{n}|} = s_2^*$ \textrm{and} $  \frac{|I_1^{n+1}|}{|I_1^{n}|}  = \frac{|I_3^{n+1}|}{|I_3^{n}|} = s_2^* $
\item $\frac{|I_1^{n+1}|}{|I_2^{n}|} = s_1^*$
\item $\frac{|I_3^{n+1}|}{|I_2^{n}|} = s_3^*$}
\end{enumerate}
\end{rem}



\begin{rem}
\ Let $I_2^n = [ y_n, z_n]$ be the interval containing $c^*$ corresponding to the scaling data  $s^*  = (s_1^*, s_2^*, s_3^*)$ then
 $$  f_{s^*}(y_n) = u_{c^*}(y_n) .$$ Hence, $f_{s^*}$ has a quadratic tip. 
\end{rem}

\begin{rem}\label{r5}
\ The invariant Cantor set of the map $f_{s^*} $ is more complex than the invariant period doubling  Cantor set of piece-wise affine infitely renormalizable map \cite{CMMT}. The complexity in the sense that:
\begin{itemize} 
\item[(i)] The geometry of Cantor set of $f_{s^*}$ is similar to the geometry of Cantor set of \cite{CMMT}, on each scale and everywhere the same scaling ratios are used,
\item[(ii)] But unlike the Cantor set of \cite{CMMT}, there are now three ratios at each scale.
\end{itemize}

 Furthermore, the scaling data corresponding to $f_{s^*} $ is very different from the geometry of Cantor attractor of analytic renormalization fixed point, in which there are no two places where the same scaling ratios are used at all scales and where the closure of the set of ratios itself a Cantor set \cite{BMT} .
\end{rem}

\subsection{$C^{1+Lip}$ extension of $f_{s^*} $} \label{extsn}
\ In subsection~\ref{p2}, we constructed a piece-wise affine infinitely renormalizable map $f_{s^*}$ corresponding to the proper scaling data $s^* .$ \\ Our goal is to construct an extension of $f_{s^*}$ in the class $C^{1+Lip}$, consisting of infinitely renormalizable maps. 
\\

 Let $ F : [0,1]^2 \rightarrow [0,1]^2 $ be the scaling function defined as
 
$$
F\left(\begin{array}{ll}
x \\ y
\end{array} 
\right) =  \left(\begin{array}{ll}
u_{c^*}(0)+s_2^*(1-x) \\ 1- s_3^*(1-y)
\end{array} 
\right)
\equiv \left(\begin{array}{ll}
F_1(x) \\ F_2(y)
\end{array} 
\right)
$$
 



 Let $K$ be the graph of $\mathcal{F}_{s^*}$  which is an extension of $f_{s^*},$ where $f_{s^*} : D_{s^*} \rightarrow [0,1]. $ Let $K^1$ and $K^2$ be the graph of $\mathcal{F}_{s^*}|_{[z_1, y_{0}]}$ and $\mathcal{F}_{s^*}|_{[z_0, y_{1}]},$ respectively, as shown in Figure \ref{fig:lipextsn}. Then $$K = \cup_{m \geq 1} F^m(K^1 \cup K^2)$$ is the graph of $\mathcal{F}_{s^*}.$ We claim that $\mathcal{F}_{s^*}$ is a $C^{1+Lip}$ unimodal map with quadratic tip.
 
  Let $B_{0} = [0,1] \times [0,1]$ and for  each $n \in \mathbb{N}, \textrm{define} $  $B_{n} = F^{n}(B_{0})$  as: \\

$ B_n=  \left\{\begin{array}{ll}
 \;[y_n,z_n] \times [\tilde{y}_{n},1], \;\;\; \textrm{if}\; n \; \textrm{is odd}  \\ \ [z_n,y_n] \times [\tilde{y}_{n},1], \;\;\; \textrm{if}\; n  \; \textrm{is even}
\end{array}
\right.$

\noindent Also, let $p_n$ be the point on the graph of the unimodal map $u_{{c}^*}(x).$ For all $n \in \mathbb{N},$ define

$ p_n=  \left\{\begin{array}{ll}
(y_{\frac{n+1}{2}},\tilde{y}_{\frac{n+1}{2}}), \;\;\; \textrm{if}\; n \; \textrm{is odd}  \\ (z_{\frac{n}{2}},\tilde{z}_{\frac{n}{2}}), \;\;\; \;\;\;\; \;\;\textrm{if}\; n  \; \textrm{is even}
\end{array}
\right.$\\
where  $ \tilde{y}_{n} = u_{c^*}(y_n)$ and $\tilde{z}_{n} = u_{c^*}(z_n).$

\begin{figure}[!htb]
\centering
{\includegraphics [width=90mm]{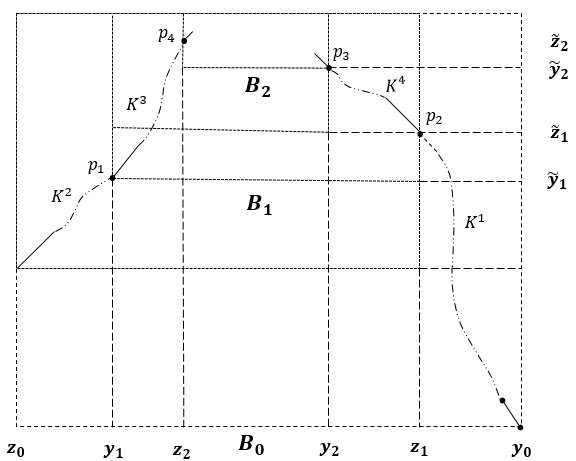}}
\caption{Extension of $f_{s^*}$}
\label{fig:lipextsn}
\end{figure} 

Then the above construction of boxes $B_n$ together with the points $p_n$ will lead to following proposition,

\begin{prop}
\ $K $ is the graph of $\mathcal{F}_{s^*}$ which is a $C^1$ extension of $f_{s^*}.$
\end{prop}
\begin{proof}
\ Since $K^1$ and $K^2$ are the graph of $f_{s^*}|_{[z_1, y_{0}]}$ and $f_{s^*}|_{[z_0, y_{1}]},$ respectively, we obtain $K^{2n+1} = F^n (K^1) $ and $K^{2n+2} = F^n (K^2) $ for each $n \in \mathbb{N}.$  \\  Note that $K^n$ is the graph of a $C^1$ function defined 
\begin{eqnarray*}
 \textrm{on} \;\; &[y_{\frac{n-1}{2}},\; z_{\frac{n+1}{2}}] \;\;\; \textrm{if} \;\; n \in 4\mathbb{N}-1, \\
 \textrm{on}\;\; &[y_{\frac{n}{2}},\; z_{\frac{n}{2}-1}] \;\;\;\;\;\textrm{if} \;\; n \in 4\mathbb{N}, \\
 \textrm{on}\;\; &[z_{\frac{n+1}{2}},\; y_{\frac{n-1}{2}}] \;\;\;\textrm{if}\;\; n \in 4\mathbb{N}+1, \\
 \textrm{and on}\;\; &[z_{\frac{n}{2}-1},\; y_{\frac{n}{2}}]\;\;\;\;\; \textrm{if}\;\; n \in 4\mathbb{N}+2.
\end{eqnarray*}
 To prove the proposition, we have to check continuous differentiability at the points $p_n.$ Consider a neighborhoods $(y_1-\epsilon,y_1+\epsilon)$ around $y_1$ and $(z_1-\epsilon,z_1+\epsilon)$ around $z_1$, the slopes are given by an affine pieces of   $f_{s^*}$ on the subintervals $(y_1,y_1+\epsilon)$ and $(z_1-\epsilon,z_1)$ and the slopes are given by the chosen $C^1$ extension on $(y_1-\epsilon,y_1)$ and  $(z_1,z_1+\epsilon).$ This implies, $K^1$ and $K^2$ are $C^1$ at $p_2$ and $p_1,$ respectively. \\ Let $\Upsilon_1 \subset K $ be the graph over the interval $(z_1-\epsilon,z_1+\epsilon)$ and $\Upsilon_2 \subset K $ be the graph over the interval $(y_1-\epsilon,y_1+\epsilon),$ \\ then the graph $K$ locally around $p_n$ is equal to $\left\{\begin{array}{ll}
F^{\frac{n-1}{2}}{(\Upsilon_2)} \;\;\; \textrm{if}\; n \; \textrm{is odd} \\ F^{\frac{n-2}{2}}{(\Upsilon_1)} \;\;\; \textrm{if}\; n  \; \textrm{is even}
\end{array}
\right.$. \\ This implies, for $n \in \mathbb{N},$  $K^{2n-1}$ is $C^1$ at $p_{2n}$ and $K^{2n}$ is $C^1$ at $p_{2n-1}.$ \\ Hence $K$ is a graph of a $C^1$ function on $[0,1] \setminus \{c^*\}.$\\ From Lemma \ref{lem24}, we observe that the horizontal contraction of $F$ is smaller than the vertical contraction. This implies that the slope of $K^n$ tends to zero when $n$ is large. \\ Therefore, $K$ is the graph of a $C^1$ function on $[0,1].$
\end{proof}

\begin{prop}
\ Let $\mathcal{F}_{s^*}$ be the function whose graph is $K$ then $\mathcal{F}_{s^*}$ is a $C^{1+Lip}$ map with a quadratic tip. 
\end{prop}

\begin{proof}
\  As the function $\mathcal{F}_{s^*}$ is a $C^1$ extension of $f_{s^*}$ and $\left.f_{s^*}\right\vert_{D_s}$ has a quadratic tip, therefore $\mathcal{F}_{s^*}$ has a quadratic tip. We have to show that $K^n$ is the graph of a $C^{1+Lip}$ function 
 $$ \mathcal{F}_{s^*}^n: Dom(K^n) \rightarrow [0,1] $$ with an uniform Lipschitz bound. \\ That is,
  for $n \geq 1,$ 
  $$Lip((\mathcal{F}_{s^*}^{n+1})') \leq Lip((\mathcal{F}_{s^*}^n)')$$ 
 Let us assume that $\mathcal{F}_{s^*}^n$ is $C^{1+Lip}$ with Lipschitz constant $\lambda_n$ for its derivatives. We show that $\lambda_{n+1} \leq \lambda_{n}$. \\
 For given $(a,b)$ on the graph of $\mathcal{F}_{s^*}^n,$ there is $(\bar{a},\bar{b}) = F(a,b)$ on the graph of $\mathcal{F}_{s^*}^{n+1}$, this implies $$\mathcal{F}_{s^*}^{n+1}(\bar{a}) = 1 -s_3^*+s_3^* \cdot \mathcal{F}_{s^*}^n(a)$$
 Since $a = 1-\frac{\bar{a}-u_{c^*}(0)}{{s_2^*}},$ we have

$$\mathcal{F}_{s^*}^{n+1}(\bar{a}) = 1 -s_3^*+s_3^* \cdot \mathcal{F}_{s^*}^n\left(1-\frac{\bar{a}-u_{c^*}(0)}{{s_2^*}}\right)$$ Differentiate both sides with respect to $\bar{a}$, we get
$$(\mathcal{F}_{s^*}^{n+1})'(\bar{a}) = \frac{-s_3^*}{{s_2^*}}(\mathcal{F}_{s^*}^n)'\left(1-\frac{\bar{a}-u_{c^*}(0)}{{s_2^*}}\right)$$

\noindent Therefore,
\begin{eqnarray}\label{eqlp}
\fl |(\mathcal{F}_{s^*}^{n+1})'(\bar{a}_1) - (\mathcal{F}_{s^*}^{n+1})'(\bar{a}_2)| &= \bigg|\frac{-s_3^*}{{s_2^*}}\bigg|\cdot \bigg|(\mathcal{F}_{s^*}^n)'\left(1-\frac{\bar{a}_1-u_{c^*}(0)}{{s_2^*}}\right)-(\mathcal{F}_{s^*}^n)'\left(1-\frac{\bar{a}_2-u_{c^*}(0)}{{s_2^*}}\right)\bigg| \nonumber\\ 
&\leq \frac{s_3^*}{({s_2^*})^2} \cdot \lambda(\mathcal{F}_{s^*}^n)'|\bar{a}_1-\bar{a}_2|
\end{eqnarray}
 
\noindent From Lemma \ref{lem24}, we have $({s_2^*})^2 = {s_3^*}.$ Therefore,
 $$\lambda(\mathcal{F}_{s^*}^{n+1})' \leq \lambda(\mathcal{F}_{s^*}^{n})' \leq \lambda(\mathcal{F}_{s^*}^{1})'$$
 This completes the proof.   
\end{proof}

\noindent Note that $f_{s^*}$ is infinitely renormalizable piece-wise affine map and $\mathcal{F}_{s^*}$ is the $C^{1+Lip}$ extension of $f_{s^*}$ which is not a $C^2$ map . Then it implies that $\mathcal{F}_{s^*}$ is renormalizable. We observe that $R \mathcal{F}_{s^*}$ is an extension of $Rf_{s^*}.$ Therefore $R \mathcal{F}_{s^*}$ is renormalizable. Hence, $\mathcal{F}_{s^*}$ is infinitely renormalizable map. Then we have the following theorem,

\begin{thm}\label{thm:ext}
\ There exists a period tripling infinitely renormalizable $C^{1+Lip}$ unimodal map $\mathcal{F}_{s^*}$ with a quadratic tip such that $$ R \mathcal{F}_{s^*} = \mathcal{F}_{s^*} .$$
\end{thm}

\subsection{Topological entropy of renormalization} \label{entropy}

\ In this section, we calculate the topological entropy of period tripling renormalization operator. \\ Let us consider three $C^{1+Lip}$ maps $\psi_0 : [0,y_1] \cup [z_1,1] \rightarrow [0,1], $  $\psi_1 : [0,y_1] \cup [z_1,1] \rightarrow [0,1],$ and $\psi_2 : [0,y_1] \cup [z_1,1] \rightarrow [0,1] $ which extend $f_{s^*}.$ For a sequence $\alpha = \{\alpha_n\}_{n \geq 1} \in \Sigma_3,$ \\ where $\Sigma_3 =  \{ \{x_n\}_{n \geq 1} : x_n \in \{0,\;1,\;2\} \}$ is called full 3-Shift. \\ Now define $$K^n(\alpha) = F^n(graph \; \psi_{{\alpha}_n}),$$ we have $$K(\alpha) =  \bigcup\limits_{n \geq 1} K^{n}(\alpha) .$$ Therefore, we conclude that $K(\alpha)$ is the graph of a $C^{1+Lip}$ map $f_\alpha$ having quadratic tip by using the same facts of subsection \ref{extsn}. \\ The shift map $\sigma : \Sigma_3 \rightarrow \Sigma_3$ is defined as $$\sigma(\alpha_1 \alpha_2 \alpha_3 \ldots) = (\alpha_2 \alpha_3 \alpha_4 \ldots).$$ 

\begin{prop}
The map $f_\alpha^3 : [y_1,\;z_1] \rightarrow [y_1,\;z_1] $ is a unimodal map for all $\alpha \in \Sigma_3.$ Furthermore, $f_\alpha$ is period tripling renormalizable and $R f_\alpha = f_{\sigma(\alpha)}.$
\end{prop}
\begin{proof}
\ We know that $f_\alpha :  [y_1,\;z_1] \rightarrow I_3^1 $ is a unimodal and onto, $f_\alpha :  I_3^1 \rightarrow I_1^1 $ is onto and affine and also $f_\alpha : I_1^1 \rightarrow  [y_1,\;z_1]  $ is onto and affine. Therefore $f_\alpha$ is renormalizable. The above construction implies $$R f_\alpha = f_{\sigma(\alpha)}.$$
\end{proof}

\noindent This gives us the following theorem.

\begin{thm} \label{prep2}
\ The period tripling renormalization operator $R$ defined on the space of $C^{1+Lip}$ unimodal maps has positive topological entropy.
\end{thm}
\begin{proof}
\ From the above construction, we conclude that $\alpha \longmapsto f_\alpha \in C^{1+Lip}$ is injective.  The domain of $R$ contains a subset $\Lambda$ on which $R$ is topological conjugate to the full 3-shift. As topological entropy $h_{top}$ is an invariant of topological conjugacy. Hence $h_{top}(\left.R\right\vert_{\Lambda} ) \geq \ln 3.$ 
\end{proof}

In fact, the topological entropy of period tripling operator $R$ on $C^{1+Lip}$ unimodal maps is unbounded because if we choose $n$ different $C^{1+Lip}$ maps, say, $\psi_0,\;\psi_1,\;\psi_2,\ldots \psi_{n-1},$ which extends $f_{s^*},$ then it will be embedded a full $n-\textrm{shift}$ in the domain of $R$. 

\subsection{An $\epsilon-$variation of the scaling data}\label{evar}
\ In this section, we use an $\epsilon$ variation on the construction of scaling data as presented in subsection ~\ref{p2} to obtain the following theorem.

\begin{thm}
\ There exists a continuum of fixed points of period tripling renormalization operator acting on $C^{1+Lip}$ unimodal maps.
\end{thm}

\begin{proof}
 Consider an $\epsilon$ variation on scaling data and we modify the construction which is described in subsection ~\ref{p2}. This modification is illustrated in Figure \ref{fig:ChaoScale}. 

Define a neighborhood $N_{\epsilon}$ about the point $(u_{c}^2(0),\; u_{c}^3(0))$  as 
$$N_{\epsilon}(u_{c}^2(0),\; u_{c}^3(0)) = \{ (u_{c}^2(0),\; \epsilon \cdot u_{c}^3(0)) \; : \; \epsilon > 0 \; \textrm{ and} \; \epsilon \; \textrm {close to } 1  \}.$$  
\FloatBarrier
\begin{figure}[h]
\centering
{\includegraphics [width=90mm]{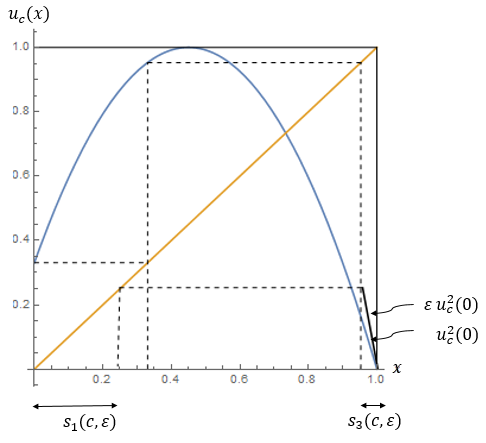}}
\caption{$\epsilon$- perturbation of $s_3(c).$ }
\label{fig:ChaoScale}
\end{figure} 
\FloatBarrier

From above Figure \ref{fig:ChaoScale}, the scaling data is obtained as
\begin{eqnarray*}
s_3(c,\epsilon) &= 1- u_{c}^2(0) \\ 
s_1(c,\epsilon) &= \epsilon \cdot u_{c}^3(0) \\
s_2(c,\epsilon) &= u_{c}(\epsilon \cdot u_{c}^3(0)) - u_{c}(0)
\end{eqnarray*}
where $c \in (0, \frac{1}{2}).$ Also, we define
\begin{eqnarray*}
\mathcal{R}(c,\epsilon) &= \frac{u_{c}(\epsilon \cdot u_{c}^3(0)) - c}{s_2(c,\epsilon)}.
\end{eqnarray*}

From subsection~\ref{p2}, we know that the period tripling renormalization operator $\mathcal{R}$ has unique fixed point $c^*$. Consequently, for each $\epsilon,$ $\mathcal{R}(c,\epsilon)$ has only one unstable fixed point, namely $c_{\epsilon}^*.$ Therefore, we consider the perturbed scaling data ${s}_{\epsilon}^* : \mathbb{N} \rightarrow T_3$ with 
 \begin{eqnarray*}
 {s}_{\epsilon}^*(n) = (\epsilon \cdot u_{c_{\epsilon}^*}^3(0),\; u_{c_{\epsilon}}(\epsilon \cdot u_{c_{\epsilon}^*}^3(0))-u_{c_{\epsilon}^*}(0),\; 1-u_{c_{\epsilon}^*}^2(0)).
 \end{eqnarray*}
 
\noindent Then $\sigma(s_{\epsilon}^* ) = {s}_{\epsilon}^* $ and using Lemma \ref{lem1}, we have
   $$Rf_{{s}_{\epsilon}^*} = f_{{s}_{\epsilon}^*}. $$
   
Moreover, $f_{{s}_{\epsilon}^*}$ is a piece-wise affine map which is infinitely renormalizable. Now we use similar extension described in subsection~\ref{extsn}, then we get $\mathcal{F}_{{s}_{\epsilon}^*}$ is the $C^{1+Lip}$ extension of $f_{{s}_{\epsilon}^*}$. This implies that $\mathcal{F}_{{s}_{\epsilon}^*}$ is a renormalizable map. As $R \mathcal{F}_{{s}_{\epsilon}^*}$ is an extension of $Rf_{{s}_{\epsilon}^*}.$ Therefore $R \mathcal{F}_{{s}_{\epsilon}^*}$ is renormalizable. Hence, for each $\epsilon$ close to $1,$ $\mathcal{F}_{{s}_{\epsilon}^*}$ is a fixed point of period tripling renormalization. This proves the existence of a continuum of fixed points of renormalization.  
\end{proof}

\begin{rem}\label{ob1}
\begin{itemize}
\item[(i)] When $\epsilon = 1,$ the fixed point of $\mathcal{R}(c,\epsilon)$ coincides with the fixed point of $\mathcal{R}(c)$ which is described in subsection \ref{p2}.
\item[(ii)] On the other hand, we have the following relations 
\begin{itemize}
\item[(i)] if $\epsilon < 1,$ then 
   $$ c^* < c_{\epsilon}^*,  $$
\item[(ii)]
    if $\epsilon > 1,$ then 
   $$ c_{\epsilon}^* <  c^*,  $$
\item[(iii)] if $\epsilon_0 < \epsilon_1,$ then 
   $$ c_{\epsilon_1}^* < c_{\epsilon_0}^*,  $$
\end{itemize} 

\end{itemize}
\end{rem}

The above relations $(i)$ and $(ii)$ are illustrated in Figure \ref{fig:cplt} by plotting the graphs of $\mathcal{R}(c,\epsilon)$ as a thick line and $\mathcal{R}(c)$ as a dotted line.
\begin{figure}[htb]
  \centering
  \begin{subfigure}[b]{0.5\linewidth}
    \centering\includegraphics[scale=0.6]{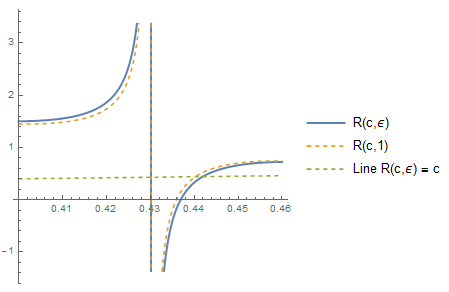}
    \caption{\label{fig:cfig1}}
  \end{subfigure}%
  \begin{subfigure}[b]{0.5\linewidth}
    \centering\includegraphics[scale=0.6]{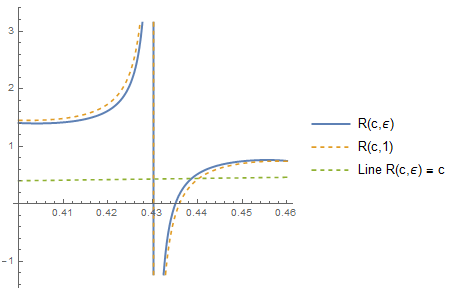}
    \caption{\label{fig:cfig2}}
  \end{subfigure}
  \caption{(\subref{fig:cfig1}): $\mathcal{R}(c, \epsilon)$ and $\mathcal{R}(c)$ for $\epsilon < 1$ and (\subref{fig:cfig2}): $\mathcal{R}(c, \epsilon)$ and $\mathcal{R}(c)$ for $\epsilon > 1.$ }
 
  \label{fig:cplt}
\end{figure}


\begin{Lem}\label{HorL}
\ There exists $\rho > 0$ such that  $$\frac{1}{\rho} \leq \frac{|\hat{I}_2^n|}{|I_2^n|^2}  \leq \rho.$$
\end{Lem}
\begin{proof}
\ We have 

$I_2^n=  \left\{\begin{array}{ll}
\;[y_n, z_n], \;\;\; \textrm{for} \;\; n = 1,3,5,... \\ \ [z_n, y_n], \; \;\;\textrm{for} \;\; n = 2,4,6,... 
\end{array}
\right.$ 

and
 
$I_3^n=  \left\{\begin{array}{ll}
\;[w_n, y_{n-1}], \;\;\; \textrm{for} \;\; n = 1,3,5,... \\ \ [y_{n-1}, w_n], \; \;\;\textrm{for} \;\; n = 2,4,6,... 
\end{array}
\right.$


where $y_n$ $z_n$ and $w_n$ are defined in subsection \ref{p2}.

\begin{figure}[!htb]
\centering
{\includegraphics [width=120mm]{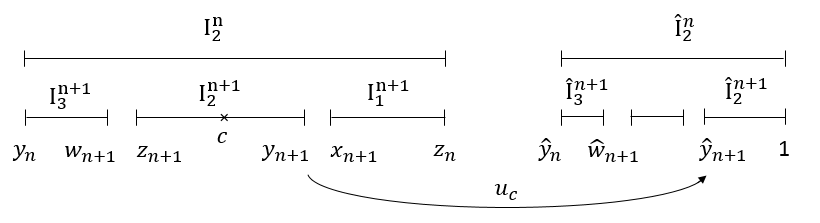}}
\caption{}
\label{fig:hor1}
\end{figure}

Let

$$
\hat{I}_2^n = u_c({I}_2^n ) = [u_c(y_n),1]  =  [\hat{y}_n,1].
$$
and  $$\hat{I}_3^n = [\hat{y}_n,\hat{w}_{n+1}] \subset \hat{I}_2^n $$ such that $$|\hat{I}_3^{n+1}| = s_2(n) |\hat{I}_2^n|.$$
Define an affine homeomorphism $f_{\alpha} : {I}_3^{n+1} \rightarrow \hat{I}_3^{n+1}$ such that $$f_\alpha(y_n) = u_c(y_n) = \hat{y}_n.$$
Also, we have $c(n) = c(\sigma^n\alpha) \in [c_0^*,c_2^*]$ which is a small  neighborhood of $c^*.$ This implies, there exists $\rho_0  > 0$ such that 
 \begin{equation}
 \frac{1}{\rho_0} \leq \frac{|c-y_n|}{|I_2^n|} \leq  \rho_0. \label{eqh1}
 \end{equation}
  Then 
\begin{eqnarray}
\frac{|\hat{I}_2^n|}{|I_2^n|^2} =  \frac{|[u_c(y_n),1]|}{|I_2^n|^2} = \frac{1}{|I_2^n|^2}\frac{(y_n-c)^2}{(1-c)^2} \label{eqh2}
\end{eqnarray} 
From Eqns. (\ref{eqh1}) and (\ref{eqh2}), we have 
\begin{eqnarray*}
 \frac{1}{\rho_0^2 (1-c)^2} \leq \frac{|\hat{I}_2^n|}{|I_2^n|^2}  \leq \frac{\rho_0^2}{(1-c)^2}. 
\end{eqnarray*}
Since $(1-c)^2 \leq \frac{1}{(1-c)^2}$, choose $\rho = \frac{\rho_0^2}{(1-c)^2}.$ Then we have 
$$ \frac{1}{\rho} \leq \frac{|\hat{I}_2^n|}{|I_2^n|^2}  \leq \rho. $$
\end{proof}

\begin{thm}
\ There exists an infinitely renormalizable $C^{1+Lip}$ unimodal map $k$ with quadratic tip such that $\{c_n\}_{n \geq 0},$ where $c_n$ is the critical point of $R^nk,$ is dense in a Cantor set.
\end{thm}
\begin{proof}
From subsection \ref{p2}, we conclude that the map $R$ has one fixed point $c^*$ which is expanding. We choose $\epsilon_0 >\epsilon_1 > \epsilon_2$ close to $1.$ Then $R(c,\epsilon_0),$ $R(c,\epsilon_1)$ and $R(c,\epsilon_2)$ have the expanding fixed points $c_0^*,$ $c_1^*$ and $c_2^*,$ respectively. In fact, $$ \frac{\partial R}{\partial c}(c_i^*,\epsilon_i) > 2, \; \; \; \textrm{for each}\; i\; = \;0,\;1,\; 2.$$ From Remark \ref{ob1}$(iii)$, we have $c_0^* < c_1^*,  <c_2^*. $ Therefore, there exists three intervals $A_0 = [c_0^*, \;a_0],$ $A_1 = [a_1, \;b_1]$ and $ A_2 = [a_2,\; c_2^*],$ such that the maps $$R_i : A_i  \rightarrow [c_0^*,\;c_2^*] \supset A_i, \;\; \;\; \textrm{for}\; i=0,1,2,$$ 
are expanding diffeomorphisms, 
where $R_i(c) = R(c,\epsilon_i).$ 
 Therefore, we get a horseshoe map. 
We use the following coding for the invariant Cantor set of the horseshoe map $$ c : \Sigma_3 \rightarrow [c_0^* , c_2^*] $$ with $$ c\left(\sigma(\alpha)\right) =  \mathcal{R}\left(c(\alpha),  \epsilon_{\alpha_0} \right).$$
Given a sequence $\alpha \in \Sigma_3,$ the proper scaling data $ s: \mathbb{N}  \rightarrow \Delta^3 $ is defined as 
$$s(n) = \left(  s_1(c(\sigma^n \alpha),\epsilon_{\alpha_n}),\; s_2(c(\sigma^n \alpha),\epsilon_{\alpha_n}),\; s_3(c(\sigma^n \alpha),\epsilon_{\alpha_n})   \right).$$ Consequently, we define a piece-wise affine map $$ f_\alpha : D_s \rightarrow [0,1].$$
Define $ F^n : [0,1]^2 \rightarrow [0,1]^2 $ by
 
$$
F^n\left(\begin{array}{ll}
x \\ y
\end{array} 
\right) =  \left(\begin{array}{ll}
F_1^n(x) \\ F_2^n(y)
\end{array} 
\right)
$$
where $F_2^n$ be the affine orientation preserving homeomorphism and $F_1^n$ be the affine homeomorphism with $F_1^n(1) = y_n.$ Therefore, the image of $F^n$ is $B_n.$     

\begin{figure}[!htb]
\centering
{\includegraphics [width=70mm]{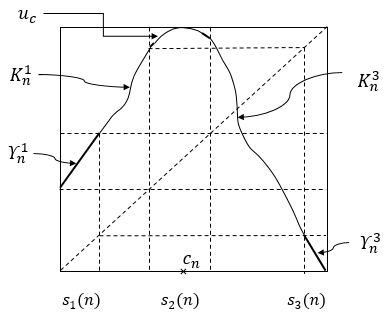}}
\caption{}
\label{fig:horext}
\end{figure}

Let $K^n = (F^n)^{-1}(graph\;f_\alpha) = \gamma_n^1 \cup \gamma_n^3.$ This is the graph of $f_{\alpha_n}.$ We extend the function $f_{\alpha_n}$ and its graph  $F_n$ on the gaps $ [u_{c_n}^3(z_{n-1}), u_{c_n}(z_{n-1})]$ and  $[u_{c_n}^4(z_{n-1}),u_{c_n}^2(z_{n-1})].$  Obverse that $u_{c_n}^i(z_{n-1})$ and $ Df_{\alpha_n}(u_{c_n}^i(z_{n-1}))$ vary within a compact family. This allows us to choose from a compact family of $C^{1+Lip}$ diffeomorphisms the extensions $$ k_n^1 : [z_{n-1}, u_{c_n}(z_{n-1})] \rightarrow [f_{\alpha_n}(z_{n-1}), f_{\alpha_n}(u_{c_n}(z_{n-1}))]$$ and 
$$ k_n^3 : [ u_{c_n}^4(z_{n-1}),1] \rightarrow [0, f_{\alpha_n}(u_{c_n}^4(z_{n-1}))]$$
of the maps $f_{\alpha_n}$  on $I_1^n$ and $I_3^n,$ respectively. Let $K_n^1$ and $K_n^3$ are the graphs of  $k_n^1$ and $k_n^3,$ respectively. Therefore,
$$K =   \bigcup\limits_{n \geq 0} F^n(K_n^1 \cup K_n^3).$$ Then $K$ is the graph of a unimodal map $k : [0,1] \rightarrow [0,1]$ which extends $f_\alpha.$ Notice that $k $ is $C^1.$ Since $f_\alpha$ has a quadratic tip, $k$ also has a quadratic tip. Also, $F^n(K^n)$ is the graph of a $C^{1+Lip}$ diffeomorphism. For a similar reason as of Eqn. (\ref{eqlp}) in  subsection \ref{extsn}, the Lipschitz bound $\lambda$ satisfies $$\lambda_{n} \leq \frac{|\hat{I}_2^n|}{|I_2^n|^2} P_0 \;\; \textrm{for some} \;\; P_0>0. $$ Using Lemma \ref{HorL}, we get $$  \lambda \leq P_0 \cdot \rho.$$ 
Thus $k_\alpha$ is a $C^{1+Lip}$ unimodal map with quadratic tip. The construction implies that $k$ is infinitely renormalization and $$graph(R^nk_\alpha) \supset K^n.$$ By choosing $\alpha \in \Sigma_3$ such that the orbit under the shift $\sigma$  is dense in the invariant Cantor set of the horseshoe map.

\end{proof}
\newpage
\section{Period quintupling renormalization}\label{sec3}
\ In this section, we describe the construction of period quintupling renormalizable fixed point on the space of piece-wise affine infinitely renormalizable maps. Further, we \\ describe the extension of piece-wise affine renormalizable map to a $C^{1+Lip}$ unimodal map. Then, we discuss the entropy of period quintupling renormalization operator acting on the same space.\\

\noindent Consider $T_k$ as defined in subsection \ref{p2},  for $k = 5,$ each element $ (s_1, s_2, s_3, s_4, s_5) $ is called a scaling quint-factor. 
Define affine maps $\tilde{s}_1,$ $\tilde{s}_2,$ $\tilde{s}_3,$ $\tilde{s}_4$ and $\tilde{s}_5$ induced by a scaling quint-factor, $$ \tilde{s}_i   : [0,1]\rightarrow [0,1] \;\;\; \textrm{for} \; i = 1,2,\cdots,5, $$  
defined by 
\begin{eqnarray*}
 \tilde{s}_1(t) &=  s_1(1-t) \\ 
 \tilde{s}_2(t) &=  u_c^3(0)+ s_2(1-t) \\
 \tilde{s}_3(t) &=  u_c(0)+ s_3(1-t) \\
 \tilde{s}_4(t) &=  u_c^2(0)- s_4(1-t) \\
 \tilde{s}_5(t) &=  1- s_5(1-t) 
\end{eqnarray*}
\ The scaling quint-factor $ s(n) = (s_1(n),s_2(n), s_3(n)), s_4(n)), s_5(n)) \in T_5 $ induces a quintuplet of affine maps $ (\tilde{s}_1(n) , \tilde{s}_2(n) , \tilde{s}_3(n), \tilde{s}_4(n), \tilde{s}_5(n))$.   \\
For each $n \in \mathbb{N}$ and $i \in \{1,2,..,5\},$ we define the following intervals:
\begin{eqnarray*}
I_i^n = \tilde{s}_2(1)\circ \tilde{s}_2(2)\circ \tilde{s}_2(3) \circ.....\circ \tilde{s}_2(n-1)\circ \tilde{s}_i(n)([0,1]).
\end{eqnarray*}

\noindent Given a proper scaling data, define $$\{c\} = \bigcap\limits_{n \geq 1} I_2^n. $$

\noindent  A proper scaling data $ s : \mathbb{N} \rightarrow T_5 $ induces the set $ D_s  = \bigcup\limits_{n \geq 1} (I_1^n \cup I_3^n \cup I_4^n \cup I_5^n) .$   Consider a map $$ g_s : D_s \rightarrow [0,1] $$ such that $g_s|_{I_i^n}$ is  the affine extensions of  $u_c|_{\partial I_i^n},$ for each $i \in \{1,3,4,5\}.$

\FloatBarrier
\begin{figure}[h]
\centering
{\includegraphics[scale=0.7]{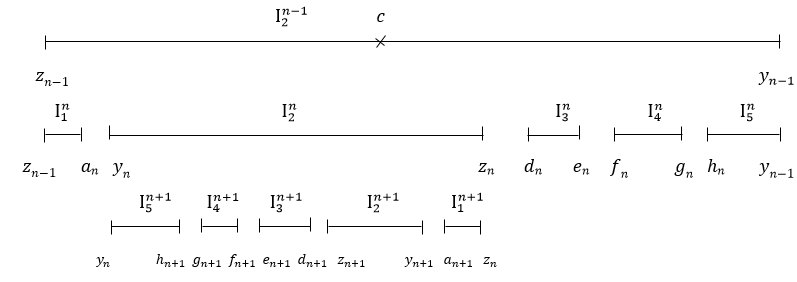}}
\caption{The intervals of next generations}
\label{fig:qnextgen}
\end{figure} 
\FloatBarrier

 \noindent From Figure~\ref{fig:qnextgen}, the end points of the intervals at each level are denoted by

 $$z_0 = 0,\; y_{0} = 1,\; I_2^0 = [0,1] $$ and for $ n \geq 1,$ 
\begin{center}
$ a_n =  \partial I_1^n \backslash \partial I_2^{n-1} $ \vspace{0.2cm} \\
$ y_{2n-1} = min\{ \partial I_2^{2n-1} \}, $ \;\;\;\;\;\;\;\;
$ y_{2n} = max\{ \partial I_2^{2n} \} $\\
$ z_{2n-1} = max\{ \partial I_2^{2n-1} \} , $ \;\;\;\;\;\;\;\;
$ z_{2n} = min\{ \partial I_2^{2n} \} $\\
$ d_{2n-1} = min\{ \partial I_3^{2n-1} \} , $ \;\;\;\;\;\;\;\;
$ d_{2n} = max\{ \partial I_3^{2n} \} $\\
$ e_{2n-1} = max\{ \partial I_3^{2n-1} \} , $ \;\;\;\;\;\;\;\;
$ e_{2n} = min\{ \partial I_3^{2n} \} $\\
$ f_{2n-1} = min\{ \partial I_4^{2n-1} \} , $ \;\;\;\;\;\;\;\;
$ f_{2n} = max\{ \partial I_4^{2n} \} $\\
$ g_{2n-1} = max\{ \partial I_4^{2n-1} \} , $ \;\;\;\;\;\;\;\;
$ g_{2n} = min\{ \partial I_4^{2n} \} $  \vspace{0.2cm}  \\
$ h_n =  \partial I_5^n \backslash \partial I_2^{n-1} $
\end{center}

\begin{defn}
\ For a given proper scaling data $s : \mathbb{N} \rightarrow T_5, $ a map $g_s$ is said to be \textit{period quintupling infinitely renormalizable} if for $n \geq 1, $ 
\begin{itemize}
\item[(1)]  $[g_s(y_{n}), \;1]$ is the maximal domain containing $1$ on which $g_s^{5^n-1}$ is defined affinely and $[0,\; g_s^2(y_{n})]$ is the maximal domain containing $0$ on which $g_s^{5^n-2}$ is defined affinely, 
\item[(2)] $[g_s(0),\;g_s^3(y_{n})]$ is the maximal domain on which $g_s^{5^n-3}$ is defined affinely and $[g_s^4(y_{n}),\;g_s^2(0)]$ is the maximal domain on which $g_s^{5^n-4}$ is defined affinely,
\item[(3)] $g_s^{5^n-1}([g_s(y_{n}),\;1]) = I_2^n,$ \\ $g_s^{5^n-2}([0,\; g_s^2(y_{n})]) = I_2^n,$ \\ $g_s^{5^n-3}([g_s(0),\;g_s^3(y_{n})]) = I_2^n,$ \\ $g_s^{5^n-4}([g_s^4(y_{n}),\;g_s^2(0)]) = I_2^n.$
\end{itemize}
\end{defn}

\noindent \ Define $U_\infty = \{g_s :  g_s \;\;\textrm{is a period quintupling infinitely renormalizable}\}$.

 \noindent Let $g_s \in U_\infty$ be given by the proper scaling data $s : \mathbb{N} \rightarrow T_5 $ and
  define 
  \begin{eqnarray*}
  {\hat{I}}_{2,+}^n &= [u_c( y_{n}),1] \;\;\;\;\;\;\;\;\;\;\;\; = [g_s( y_{n}),1], \\
  {\hat{I}}_{2,++}^n &= [0, u_c^{2}( y_{n})] \;\;\;\;\;\;\;\;\;\;\;\; = [0, g_s^{2}( y_{n})],   \\
   {\hat{I}}_{2,-}^n & = [u_c^{-1}( z_{n}),u_c^{-1}( y_{n})] = [g_s^{-1}( z_{n}),g_s^{-1}( y_{n})] ,\\
  {\hat{I}}_{2,--}^n  &= [max\{u_c^{-2}( y_{n})\},max\{u_c^{-2}( z_{n})\}] \\ &= [max\{g_s^{-2}( y_{n})\},max\{g_s^{-2}(z_{n})\}] 
\end{eqnarray*}

\noindent  where $ u_c^{-1}(y)$ denotes the preimage(s) of $y$ under $u_c$ and  $ u_c^{-2}(y)  = \{ x \in I_2^0:\; u_c^2(x)=y \}.$
   \\
   
 \noindent  Let $$ h_{s,n} : [0,1] \rightarrow [0,1] $$ be defined by $$h_{s,n} = \tilde{s}_2(1)\circ \tilde{s}_2(2) \circ \tilde{s}_2(3)\circ..... \circ \tilde{s}_2(n). $$
Furthermore, let
\begin{eqnarray*}
\hat{h}_{s,n}^+ : [0,1] \rightarrow \hat{I}_{2,+}^n, \;\;\;\;\;\;
 \hat{h}_{s,n}^{++} : [0,1] \rightarrow \hat{I}_{2,++}^n , \\
\hat{h}_{s,n}^- : [0,1] \rightarrow \hat{I}_{2,-}^n, \;\;\;\;\;\;
 \hat{h}_{s,n}^{--} : [0,1] \rightarrow \hat{I}_{2,--}^n  
\end{eqnarray*}
 be the affine orientation preserving homeomorphisms. 
Then define $$R_ng_s: h_{s,n}^{-1}(D_s) \rightarrow [0,1]$$  by 
$$R_ng_s(x) =  \left\{\begin{array}{ll}
R_n^{++}g_s(x) = ({\hat{h}}_{s,n}^{++})^{-1} \circ g_s^{2} \circ h_{s,n}(x) , \;\;\;\;\; \;  \; \textrm{if} \; x \in h_{s,n}^{-1}(\mathop{\cup}\limits_{n \geq 1}I_1^n) \\ 
R_n^{+}g_s(x) = ({\hat{h}}_{s,n}^{+})^{-1} \circ g_s \circ h_{s,n}(x) , \;\;\; \;\;\;\; \;  \; \textrm{if} \; x \in h_{s,n}^{-1}(\mathop{\cup}\limits_{n \geq 1}I_5^n) \\ 
R_n^{-}g_s(x) = ({\hat{h}}_{s,n}^{-})^{-1} \circ g_s^{-1} \circ h_{s,n}(x) , \;\;\;\; \; \;  \; \textrm{if} \; x \in h_{s,n}^{-1}(\mathop{\cup}\limits_{n \geq 1}I_4^n) \\ 
R_n^{--}g_s(x)  = ({\hat{h}}_{s,n}^{--})^{-1} \circ g_s^{-2} \circ h_{s,n}(x),  \;\;  \;  \; \textrm{if} \; x \in h_{s,n}^{-1}(\mathop{\cup}\limits_{n \geq 1}I_3^n)  
\end{array}
\right.$$ \\
where, $$R_n^{++}g_s: h_{s,n}^{-1}(\mathop{\cup}\limits_{n \geq 1}I_1^n ) \rightarrow [0,1] $$ 
$$R_n^{+}g_s: h_{s,n}^{-1}(\mathop{\cup}\limits_{n \geq 1}I_5^n ) \rightarrow [0,1] $$
 $$R_n^{-}g_s: h_{s,n}^{-1}(\mathop{\cup}\limits_{n \geq 1}I_4^n ) \rightarrow [0,1] $$
 and $$ R_n^{--}g_s: h_{s,n}^{-1}(\mathop{\cup}\limits_{n \geq 1}I_3^n ) \rightarrow [0,1],$$ which are illustrated in Figure~\ref{fig:renomop5}.

\FloatBarrier
\begin{figure}[h]
\centering
{\includegraphics [width=120mm]{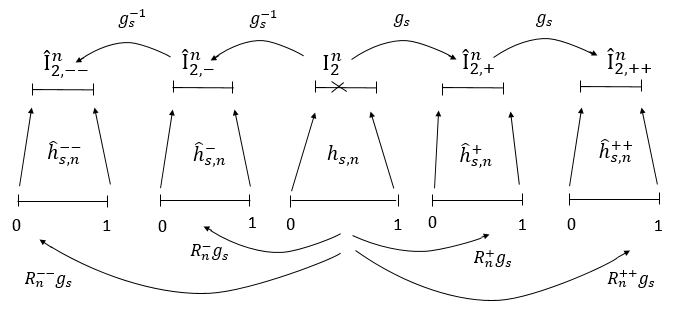}}
\caption{ }
 \label{fig:renomop5}
\end{figure}
\FloatBarrier

\begin{Lem}\label{lem51}
\ Let $s : \mathbb{N} \rightarrow T_5$ be proper scaling data such that $g_s$ is infinitely renormalizable. Then $$R_n{g_s} = g_{\sigma^{n}(s)}. $$
\end{Lem}

\noindent Let $g_s$ be infinitely renormalization, then for $n \geq 0,$ we have $$ g_s^{5^n}: D_s \cap I_2^n \rightarrow I_2^n $$ is well defined.  Define the renormalization $R : U_\infty \rightarrow U$ as
$$ Rg_s = h_{s,1}^{-1} \circ g_s^5 \circ h_{s,1} .$$

\noindent The maps $g_s^{5^n-1} : {\hat{I}}_{2,+}^n \rightarrow I_2^n,$ \; $g_s^{5^n-2} : {\hat{I}}_{2,++}^n \rightarrow I_2^n,$ \; $g_s^{5^n-3} : {\hat{I}}_{2,--}^n \rightarrow I_2^n,$  and $g_s^{5^n-4} : {\hat{I}}_{2,-}^n \rightarrow I_2^n,$ \; are the affine homeomorphisms, whenever $g_s \in  U_\infty$. Then

\begin{Lem}\label{lem52}
\ We have $R^n{g_s} : D_{\sigma^{n}(s)} \rightarrow [0,1]$ and $R^n{g_s} = R_n{g_s}.$
\end{Lem}

\noindent The~\cref{lem51} and~\cref{lem52}  give us the following result,

\begin{prop}
\ There exists a map $g_{{s}^*} \in U_\infty, $ where ${s}^*$ is characterized by $$Rg_{{s}^*} = g_{{s}^*}. $$
\end{prop}
\begin{proof}
\ Consider $s : \mathbb{N} \rightarrow T_5$ be proper scaling data such that $g_s  $ is infinitely renormalizable. Let $c_n$ be the critical point of $g_{\sigma^n(s)}.$

\noindent The combinatorics for period quintupling renormalization is shown in Figure \ref{fig:p5}.
 
\begin{figure}[!htb]
\centering{\includegraphics [width=100mm]{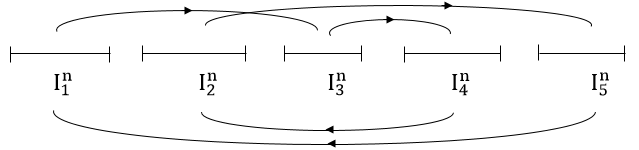}}
\caption{Period quintuple combinatorics ($ I_2^n \rightarrow I_5^n \rightarrow I_1^n \rightarrow I_3^n \rightarrow I_4^n \rightarrow I_2^n).$ }
\label{fig:p5}
\end{figure}
\begin{figure}[!htb]
\centering
{\includegraphics [width=120mm]{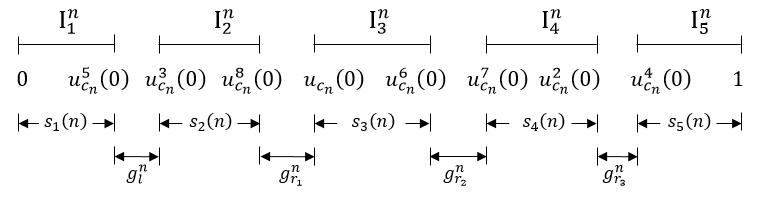}}
    \caption{Length of the intervals and gaps}
    \label{fig:qlength}
\end{figure} 

\noindent From Figures~\ref{fig:p5} and ~\ref{fig:qlength}, we have the following 
\begin{eqnarray}
 u_{c_n}^4(0) &= 1-s_5(n) \label{31}\\
 u_{c_n}^5(0) &= u_{c_n}(1-s_5(n)) \nonumber  \\ &= s_1(n) \label{32} \\
 u_{c_n}^6(0) &= u_{c_n}(s_1(n)) \nonumber  \\ &= s_3(n) +  u_{c_n}(0) \label{33} \\
 u_{c_n}^7(0) &= u_{c_n}(s_3(n)+ u_{c_n}(0)) \nonumber  \\ &= u_{c_n}^2(0)- s_4(n) \label{34} \\
 u_{c_n}^8(0) &= u_{c_n}(u_{c_n}^2(0)- s_4(n))\nonumber  \\ & = s_2(n) +  u_{c_n}^3(0) \label{35}  \\
 c_{n+1} &= \frac{u_{c_n}^8(0)-c_n}{s_2(n)} \equiv \mathcal{R}(c_n). \label{36}
\end{eqnarray}

\noindent We use Mathematica for solving the Eqns.~(\ref{31}),~(\ref{32}),~(\ref{33}),~(\ref{34}) and~(\ref{35}), and we obtain the expressions for $s_1(n),\;s_2(n),\;s_3(n),\;s_4(n)$ and $s_5(n).$ \\ Let $s_i(n) \equiv S_i(c_n)$ for $i = 1,\ldots,5.$ Since these expressions are too lengthy, we just plotted the graph of each $S_i(c).$  The graphs of $S_i(c_n)$ are shown in Figures~\ref{fig:pplots}. 

\begin{figure}[!htb]
\centering
\begin{subfigure}{.5\textwidth}
\centering
\includegraphics [width=60mm]{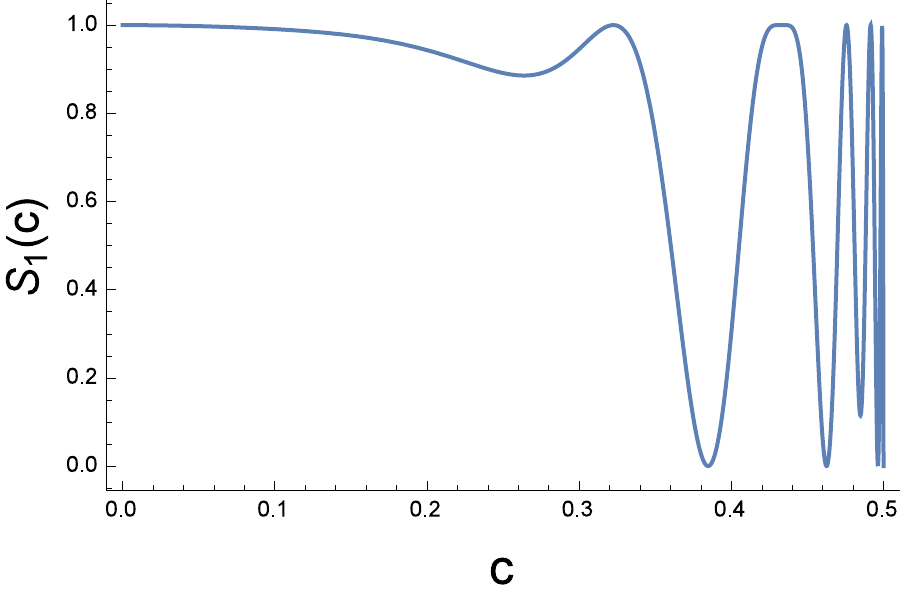}
\caption{}
\end{subfigure}%
\begin{subfigure}{.5\textwidth}
\centering
\includegraphics [width=60mm]{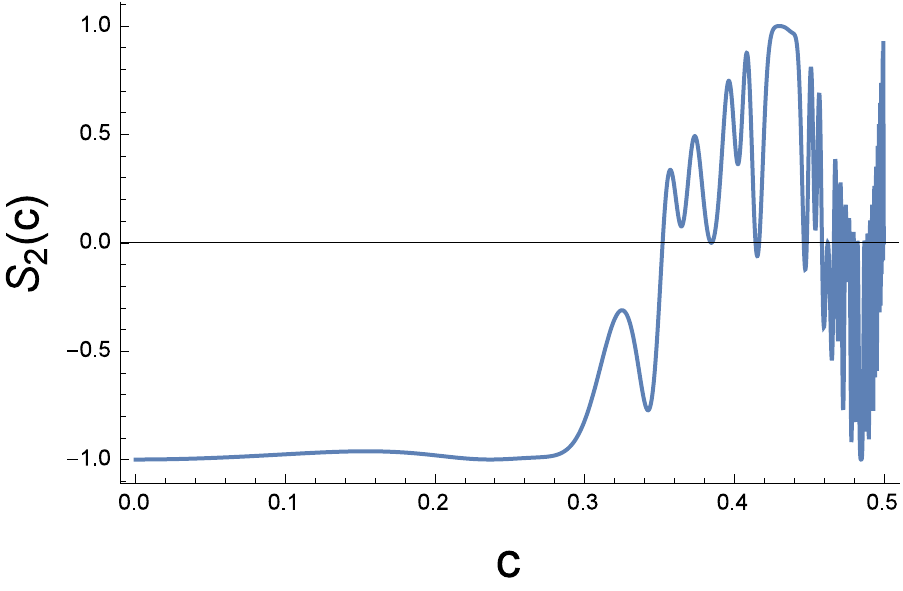}
\caption{}
\end{subfigure}

\begin{subfigure}{\linewidth}
\centering
\includegraphics [width=60mm]{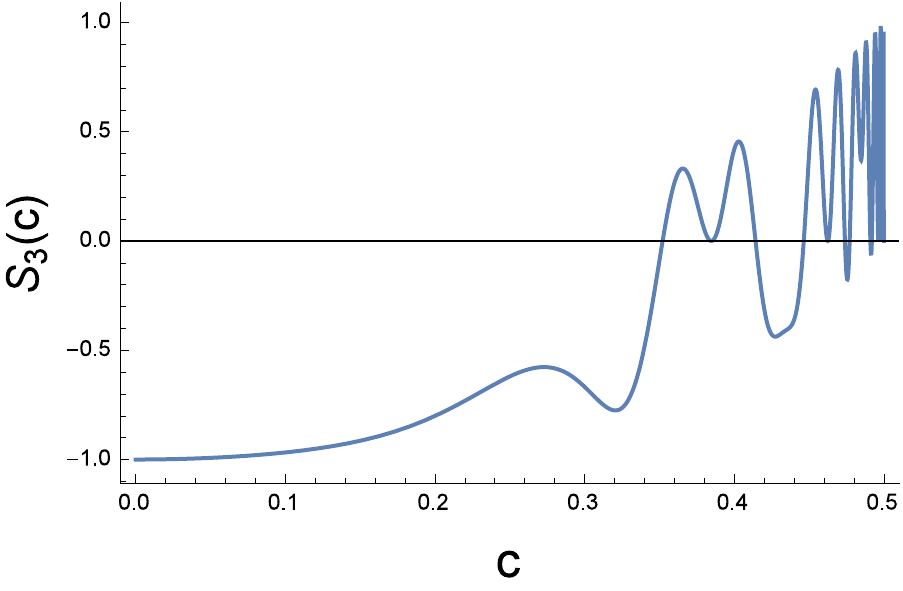}
\caption{}
\end{subfigure}
\centering
\begin{subfigure}{.5\textwidth}
\centering
\includegraphics [width=60mm]{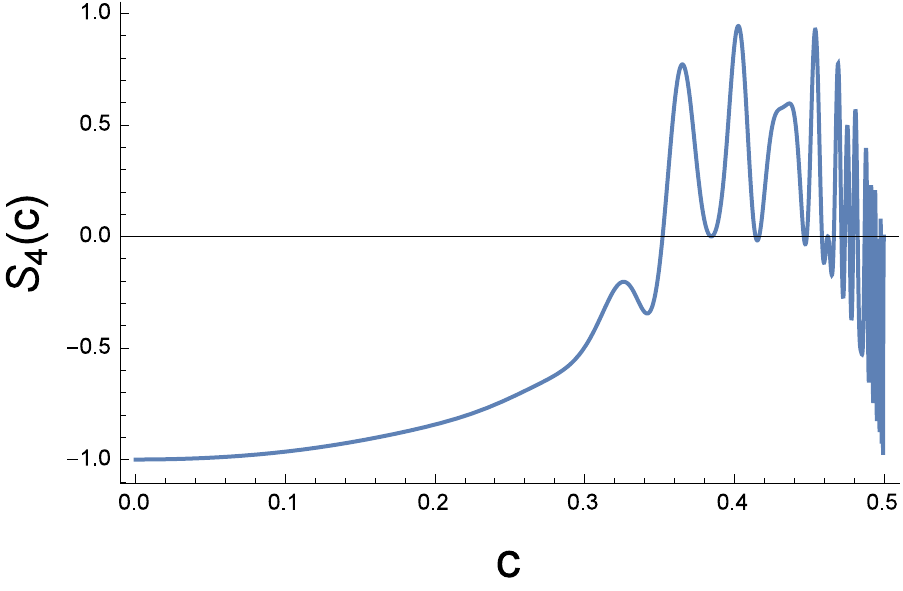}
\caption{}
\end{subfigure}%
\begin{subfigure}{.5\textwidth}
\centering
\includegraphics [width=60mm]{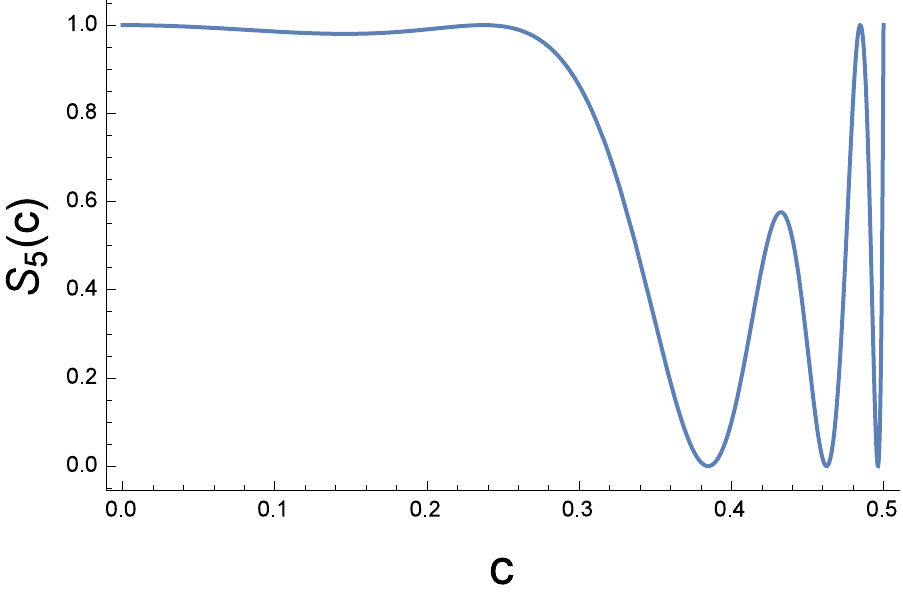}
\caption{}
\end{subfigure}
\caption{The graphs of $S_1(c),$ $S_2(c),$ $S_3(c),$ $S_4(c)$  and  $S_5(c)$}
\label{fig:pplots}
\end{figure}

\noindent Since $(s_1(n), s_2(n), s_3(n), s_4(n), s_5(n)) \in T_5,$ then we have the following conditions.
\begin{eqnarray}
 s_i(n) > 0, \textrm{ for each } i \in \{1,2,..,5\}\label{37} \\
 \sum_{i=1}^{5}s_i(n) <1 \label{38} \\
  0 < c_n < \frac{1}{2} \label{38a} .
 \end{eqnarray}

\noindent As intervals $I_i^n$ for $i = 1,2,\ldots,5$ are pairwise disjoint, let $g_{l}^n$ be the gap between  $I_1^n \; \& \; I_2^n$ and let $ g_{r_i}^n$ be the gap between  $I_{i+1}^n$ and $I_{i+2}^n$ for $i=1,2,3.$  The intervals and gaps are illustrated in Figure \ref{fig:qlength}. Then we have 
\begin{eqnarray}
g_{l}^n = u_{c_n}^3(0)-u_{c_n}^5(0)   \equiv G_{l}(c_n) > 0  \label{39}  \\
 g_{r_1}^n = u_{c_n}(0)-u_{c_n}^8(0)  \equiv G_{r_1}(c_n)> 0 \label{310}  \\
 g_{r_2}^n = u_{c_n}^7(0)-u_{c_n}^6(0)  \equiv G_{r_2}(c_n)> 0 \label{311}   \\
 g_{r_3}^n = u_{c_n}^4(0)-u_{c_n}^2(0)  \equiv G_{r_3}(c_n)> 0 \label{312}  
\end{eqnarray}
 
\noindent Note that the conditions (\ref{37}), (\ref{39}), (\ref{310}) to (\ref{312}) implies the condition (\ref{38}).

Therefore, the conditions~(\ref{37}) and (\ref{38a}) together with the gaps conditions (\ref{39}) to (\ref{312}) define the feasible domain $f_d$ to be:
  
 \begin{eqnarray}
\fl f_d = \Big\{c \in (0, 0.5)  : S_i(c) > 0 \; \forall\; i =1, \ldots ,5, \; G_l(c)>0,\; G_{r_j}(c)> 0 \; \forall\; j =1, 2,3 \Big\}. \label{fd}
\end{eqnarray}
To compute the feasible domain $f_d,$  we need to find subinterval(s) of $(0, 0.5)$ which satisfies the conditions of (\ref{fd}). By using Mathematica, we employ the following command to obtain the feasible domain 
\begin{eqnarray*} 
\fl \textup{N}[\textup{Reduce}[\{S_1(c) > 0,S_2(c) > 0,S_3(c) > 0,S_4(c) > 0,S_5(c) > 0,G_l(c) > 0,G_{r_1}(c) > 0, \\ G_{r_2}(c) > 0,G_{r_3}(c) > 0,0<c<0.5\},c]] .
\end{eqnarray*}
   This yields: 
\begin{eqnarray*}
\fl f_d = (0.379765..., 0.384772...) \cup (0.384772..., 0.390436...) \equiv f_{d_1} \cup f_{d_2}.
\end{eqnarray*}
From the Eqn.(\ref{36}), the graphs of $\mathcal{R}(c)$ in the sub-domains $f_{d_1}$ and $f_{d_2}$ of $f_d$ are shown in Figure \ref{fig:gfix}.

 \begin{figure}[ht]
  \centering
  \begin{subfigure}[b]{0.5\linewidth}
    \centering\includegraphics[width=130pt]{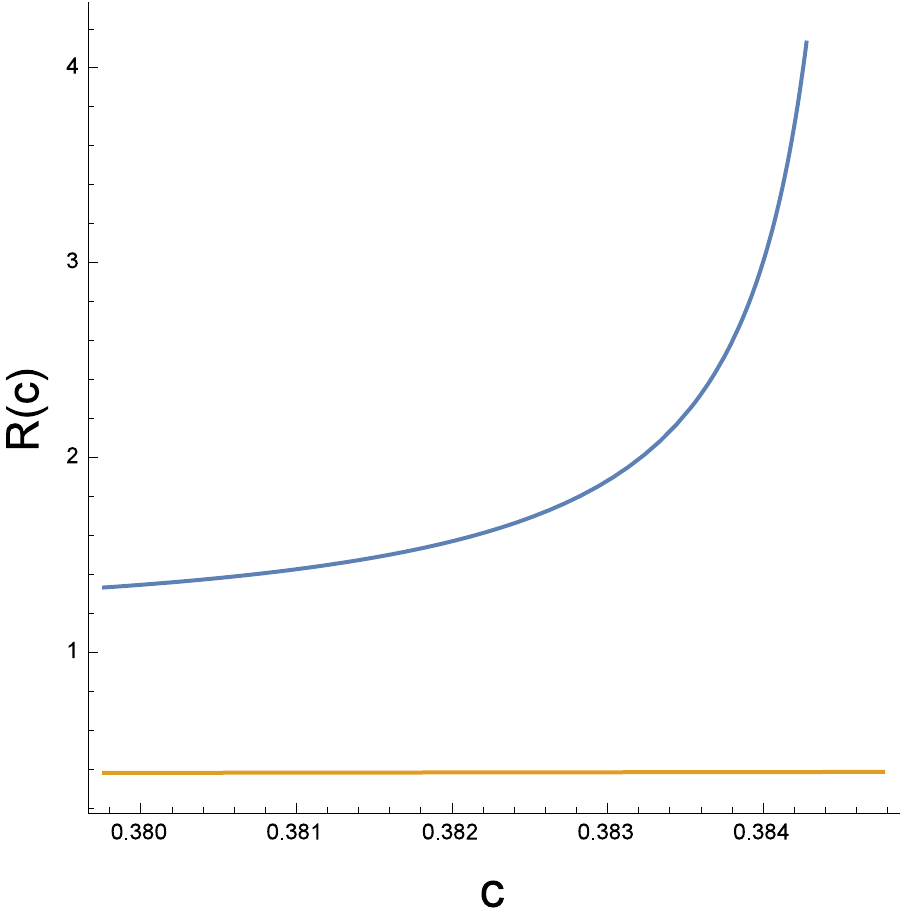}
   \subcaption{ $\mathcal{R}$ has no fixed point in $f_{d_1}.$ }
    \label{fig:gfix1}
  \end{subfigure}%
  \begin{subfigure}[b]{0.5\linewidth}
    \centering\includegraphics[width=160pt]{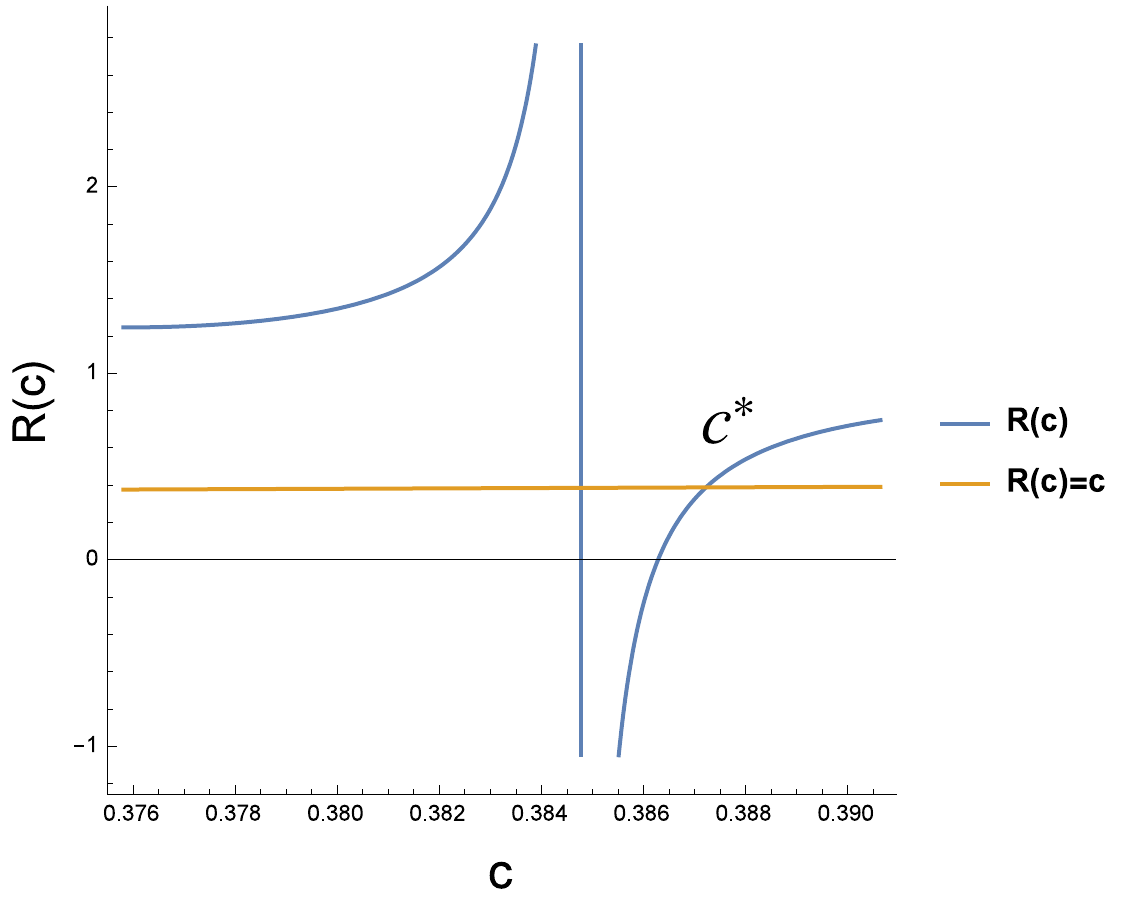}
    \subcaption{$\mathcal{R}$ has only one fixed point in $f_{d_2}.$ }
    \label{fig:gfix2}
  \end{subfigure}
  \caption{The graph of $\mathcal{R} : f_d \rightarrow \mathbb{R}$ and the diagonal $\mathcal{R}(c) = c.$}
  \label{fig:gfix}
\end{figure} 

The map $\mathcal{R}  : f_d \rightarrow \mathbb{R}$ is expanding in the neighborhood of the fixed point $c^*$ which is illustrated in Figure~\ref{fig:gfix2}. By Mathematica computations, we observe that $\mathcal{R}$ has a unique fixed point $c^* = 0.387226...$ in $f_{d_2}.$ Therefore, we conclude that $\mathcal{R}$ has only one fixed point in $f_d $ such that \\ $$\mathcal{R}(c^*) = c^*$$ corresponds to an infinitely renormalizable map $g_{s^*}.$

\noindent In other words, consider the scaling data ${s}^* : \mathbb{N} \rightarrow T_5$  with 
 \begin{eqnarray*}
 {s}^*(n) &= (s_1^*(n), s_2^*(n), s_3^*(n), s_4^*(n), s_5^*(n) )  \\
  &=(u_{c^*}^5(0), \; u_{c^*}^8(0) - u_{c^*}^3(0), \; u_{c^*}^6(0)- u_{c^*}(0), \; u_{c^*}^2(0) - u_{c^*}^7(0), \;  1-u_{c^*}^4(0)).
 \end{eqnarray*}
  Then $\sigma({s}^* ) = {s}^* $ and using Lemma \ref{lem51} we have $$Rg_{{s}^*} = g_{{s}^*}.$$
\end{proof}

\begin{Lem}
\ If $ g_{{s}^*}$ is the map with a proper scaling data $ {s}^* = (s_1^*, s_2^*, s_3^*, s_4^*, s_5^*)$ corresponding to $c^*,$ then we have
 $$(s_2^*)^2=s_5^*.$$
\end{Lem}
\begin{proof}
\ Let $\hat{I}_{2,+}^n = g_{{s}^*}(I_2^n)  = [g_{{s}^*}(y_{n}),1]$ and $\hat{I}_5^{n+1} = g_{{s}^*}(I_5^{n+1}).$ Then $g_{{s}^*}^{ 5^n-1} : \hat{I}_{2,+}^n \rightarrow I_2^n  $ is affine, monotone and onto. Further, by construction
$$ g_{{s}^*}^{ 5^n-1} : \hat{I}_{2,+}^{n+1} \rightarrow I_5^{n+1}. $$ Hence, $$\frac{|\hat{I}_{2,+}^{n+1}|}{|\hat{I}_{2,+}^n|} = s_5^*.$$ Therefore,
$|I_2^n| = (s_2^*)^n$ and $|\hat{I}_2^n| = (s_5^*)^n$. Since, $$g_{{s}^*}(y_n) = u_{{c}^*}(y_n).$$ This implies,
$$s_5^* = \frac{|\hat{I}_{2,+}^{n+1}|}{|\hat{I}_{2,+}^n|} = \left(\frac{{y_{n+1}}-c}{{y_{n}}-c}\right)^2  =\left( \frac{|{I}_2^{n+1}|}{|{I}_2^n|}\right)^2 = (s_2^*)^2.$$

\end{proof}

\begin{rem}
\ Let $I_2^n = [ y_n, z_n]$ be the interval containing $c^*$ corresponding to the scaling data  $s^*  = (s_1^*, s_2^*, s_3^*, s_4^*, s_5^*)$ then
 $$  g_{s^*}(y_n) = u_{c^*}(y_n) .$$ Hence, $g_{s^*}$ has a quadratic tip. 
\end{rem}

\begin{rem}
\ A proper scaling data $(s_1^*, s_2^*, s_3^*, s_4^*, s_5^*)$ is satisfying the following properties, for all $n \in \mathbb{N},$
\begin{enumerate}[label=(\roman*)] \normalsize{ 
\item $ \frac{|I_2^{n+1}|}{|I_2^{n}|} = s_2^*$ \textrm{and} $  \frac{|I_1^{n+1}|}{|I_1^{n}|}  = \frac{|I_3^{n+1}|}{|I_3^{n}|} = \frac{|I_4^{n+1}|}{|I_4^{n}|} = \frac{|I_5^{n+1}|}{|I_5^{n}|} = s_2^* $
\item $\frac{|I_1^{n+1}|}{|I_2^{n}|} = s_1^*$
\item $\frac{|I_3^{n+1}|}{|I_2^{n}|} = s_3^*$
\item $\frac{|I_4^{n+1}|}{|I_2^{n}|} = s_4^*$
\item $\frac{|I_5^{n+1}|}{|I_2^{n}|} = s_5^*$}
\end{enumerate}
\end{rem}

\begin{rem}\label{rq5}
\ The invariant Cantor set of the map $g_{{s}^*} $ is next in the complexity to the invariant period tripling Cantor set of piece-wise affine map $f_{{s}^*}$  which is described in subsection \ref{p2}. But unlike the Cantor set of $f_{{s}^*}$, there are now five ratios at each scale.
\end{rem}

\noindent In subsection~\ref{extsn}, we constructed $C^{1+Lip}$ extension of piece-wise affine map $f_{s^*}$ to the $C^{1+Lip}$ unimodal map $\mathcal{F}_{s^*}.$ A similar construction leads the following result.

\begin{thm}\label{thm:ext5}
\ Let $\mathcal{G}_{s^*}$ be a $C^{1+Lip}$ extension of $g_{{s}^*}.$ Then $\mathcal{G}_{s^*}$ is a period quintupling infinitely renormalizable $C^{1+Lip}$ unimodal map with a quadratic tip such that $$ R \mathcal{G}_{s^*} = \mathcal{G}_{s^*} .$$
\end{thm}



\noindent As we discussed the topological entropy in subsection \ref{entropy} and the $\epsilon-$variation on scaling data in subsection \ref{evar}, in the similar way, the following results hold for period quintupling renormalization. 

\begin{thm} \label{prep52}
\ The period quintupling renormalization defined on $C^{1+Lip}$ unimodal maps has infinite entropy and it has a continuum of fixed points.
\end{thm}

\begin{thm}
\ There exists an infinitely renormalizable $C^{1+Lip}$ unimodal map ${k}$ with quadratic tip such that $\{c_n\}_{n  \geq 0}$ is dense in a  Cantor set, where $c_n$ is the critical point of $R^nk$ and $R$ is the period quintupling renormalization operator. 

\end{thm}

\section{Conclusions}
 In this work, we showed that the period tripling and period quintupling renormalization  operators both have a fixed point corresponding to the given proper scaling data. Also, we notice that the geometry of the invariant Cantor set of the map $g_{{s}^*}$ is more complex than the geometry of the invariant Cantor set of $f_{{s}^*}.$
\; Furthermore, the piece-wise affine period tripling and quintupling renormalizable maps are extended to a $C^{1+Lip}$ unimodal map. Finally, we showed that the period tripling and period quintupling renormalizations defined on the space of $C^{1+Lip}$ unimodal maps have positive entropy. In fact, the topological entropy of renormalization operator $R$ is unbounded. We proved the existence of an infinitely renormalizable $C^{1+Lip}$ unimodal map $k$ with quadratic tip such that $\{c_n\}_{n  \geq 0}$ is dense in a  Cantor set. 
 This gives us the existence of continuum of fixed points of period tripling and period quintupling renormalizations. This shows non-rigidity of period tripling and quintupling renormalizations.

\end{document}